\documentclass{amsart}
\usepackage[english]{babel}
\usepackage[latin1]{inputenc}
\usepackage[dvips,final]{graphics}
\usepackage{amsmath,amsfonts,amssymb,amsthm,amscd,array,stmaryrd,mathrsfs}
\usepackage{pstricks}
 \usepackage[all]{xy}
 \usepackage{url}
\usepackage{textcomp}
 \usepackage[final]{epsfig}
\theoremstyle{plain}
\newtheorem{thm}{Theorem}
\newtheorem{lem}{Lemma}[section]
\newtheorem{cor}[lem]{Corollary}
\newtheorem{prop}[lem]{Proposition}

\theoremstyle{definition}

\newtheorem{rem}[lem]{Remark}
\newtheorem{ex}[lem]{Example}

\let\ssection=\section
\renewcommand{\section}{\setcounter{equation}{0}\ssection}

\newcommand{\R}{\mathbb{R}}
\newcommand{\Z}{\mathbb{Z}}
\newcommand{\C}{\mathbb{C}}

\newcommand{\A}{\mathcal{A}}

\newcommand{\cC}{\mathcal{C}}
\newcommand{\F}{\mathcal{F}}
\newcommand{\cM}{\mathcal{M}}

\newcommand{\cP}{\mathcal{P}}
\newcommand{\cZ}{\mathcal{Z}}

\newcommand{\Id}{\mathrm{Id}}
\newcommand{\SL}{\mathrm{SL}}
\newcommand{\PSL}{\mathrm{PSL}}
 
\newcommand{\Qc}{\mathcal{Q}} 
\newcommand{\Rc}{\mathcal{R}}

\newcommand{\RP}{{\mathbb{RP}}}
\newcommand{\pP}{{\mathbb{P}}}
\newcommand{\CP}{{\mathbb{CP}}}

\newcommand{\half}{\frac{1}{2}}
\newcommand{\thalf}{\frac{3}{2}}

\newcommand{\e}{\varepsilon}

\def\D{\Delta}
\def\Db{\overline{\Delta}}
\def\e{\varepsilon}

\def\om{\omega}

\begin{document}

\title[2-frieze patterns and the space of polygons]{2-frieze patterns 
and the cluster structure of\\ the space of polygons}

\author{Sophie Morier-Genoud}

\author{Valentin Ovsienko}

\author{Serge Tabachnikov}

\address{Sophie Morier-Genoud,
Institut de Math\'ematiques de Jussieu
UMR 7586
Universit\'e Pierre et Marie Curie
4, place Jussieu, case 247
75252 Paris Cedex 05
}

\address{
Valentin Ovsienko,
CNRS,
Institut Camille Jordan,
Universit\'e Claude Bernard Lyon~1,
43 boulevard du 11 novembre 1918,
69622 Villeurbanne cedex,
France}

\address{
Serge Tabachnikov,
Department of Mathematics,
Pennsylvania State University,
University Park, PA 16802, USA
}

\email{sophiemg@math.jussieu.fr,
ovsienko@math.univ-lyon1.fr,
tabachni@math.psu.edu
}

\date{}



\begin{abstract}
We study the space of 2-frieze patterns
generalizing that of the classical Coxeter-Conway frieze patterns.
The geometric realization of this space is the space of
$n$-gons (in the projective plane and in 3-dimensional vector space)
which is a close relative of the moduli space of genus $0$ curves
with $n$ marked points.
We show that the space of 2-frieze patterns is a cluster manifold and study its algebraic
and arithmetic properties.
\end{abstract}

\maketitle


\tableofcontents

\section{Introduction}

The space $\cC_n$ of $n$-gons in the projective plane
(over $\C$ or over $\R$) modulo projective equivalence is a close relative of the moduli space 
$\cM_{0,n}$ of genus zero curves
with $n$ marked points.  
The space $\cC_n$ was considered in \cite{Sch} and in \cite{OST} 
as the space on which the {\it pentagram map} acts.  

The main idea of this paper is to identify the space $\cC_n$ with the space $\F_n$ of
combinatorial objects that we call $2$-\textit{friezes}.
These objects first appeared in \cite{Pro} as generalization of the Coxeter friezes~\cite{Cox}.
We show that $\cC_n$ is isomorphic to $\F_n$, provided $n$ is not a multiple of~$3$.
This isomorphism leads to remarkable coordinate systems on $\cC_n$ and equips $\cC_n$
with the structure of cluster manifold.
The relation between $2$-friezes and cluster algebras is not surprising, since $2$-friezes can be viewed
as a particular case of famous recurrence relations known as the discrete Hirota equation,
or the octahedron recurrence. The particular case  of $2$-friezes is a very interesting subject;
in this paper we make first steps in the study of algebraic and combinatorial structures of the space of $2$-friezes.

The pentagram map $T:\cC_n\to\cC_n$, see \cite{Sch1,Sch} and also \cite{OST,Gli},
is a beautiful dynamical system which is a time and space discretization of the Boussinesq equation. Complete integrability of the pentagram map for a larger space of twisted $n$-gons was proved in \cite{OST}; recently, integrability of $T$ on $\cC_n$  was established, by different methods, in \cite{Sol} and \cite{OST2}. The desire to better understand the structure of the space of closed polygons was our main motivation.

\subsection{$2$-friezes}

We call a \textit{2-frieze pattern}
a grid of numbers, or polynomials, rational functions, etc.,
$(v_{i,j})_{(i,j)\in \Z^2}$
and $(v_{i+\half,j+\half})_{(i,j)\in \Z^2}$ organized as follows
$$
 \xymatrix{
 &
&\ar@{-}[rd]
& v_{i-\thalf,j+\thalf}\ar@{--}[ld]\ar@{--}[rd]
&\ar@<2pt>@{-}[ld]
&\\
&\ldots \ar@<2pt>@{-}[rd]
&v_{i-\thalf,j+\half}\ar@{--}[rd]\ar@{--}[ld]
& v_{i-1,j+1}\ar@{-}[ld]\ar@{-}[rd]
&v_{i-\half,j+\thalf}\ar@{--}[rd]\ar@{--}[ld]
&\ar@{-}[ld]\\ 
&v_{i-\thalf,j-\half}  \ar@{--}[rd]
& v_{i-1,j}\ar@<2pt>@{-}[rd]\ar@<2pt>@{-}[ld]
&v_{i-\half,j+\half}\ar@{--}[ld]\ar@{--}[rd]
& v_{i,j+1}\ar@<2pt>@{-}[ld]\ar@{-}[rd] 
&\ar@{--}[ld] \cdots\\
&v_{i-1,j-1} \ar@{-}[rd]
&v_{i-\half,j-\half} \ar@{--}[rd] \ar@{--}[ld]
&v_{i,j}\ar@{-}[ld]\ar@{-}[rd]
&v_{i+\half,j+\half}\ar@{--}[ld]\ar@{--}[rd]
&v_{i+1,j+1}\ar@{-}[ld]
 \\
& &v_{i,j-1}\ar@{-}[rd]
& v_{i+\half,j-\half}\ar@{--}[ld]\ar@{--}[rd]
&v_{i+1,j}  \ar@{-}[ld]&\\
&&&v_{i+1,j-1}&&&&&
}
$$
such that every entry is equal to the determinant
of the $2\times2$-matrix formed by its four neighbours:
$$
 \xymatrix{
& B\ar@{-}[ld]\ar@{-}[rd]
&F\ar@{--}[rd]&&&&\\ 
A \ar@<2pt>@{-}[rd]
&E\ar@{--}[ld]\ar@{--}[rd]\ar@{--}[lu]\ar@{--}[ru]
& D\ar@<2pt>@{-}[ld] \ar@{-}[ru]\ar@{-}[rd]&H&\ar@{=>}[r]
&&\!E=AD-BC, \quad D=EH-FG,
\quad\ldots\\
&C
&G\ar@{--}[ru]&&&&
}
$$
Generically, two consecutive rows in a 2-frieze pattern
determine the whole 2-frieze pattern.

The notion of 2-frieze pattern is a variant of
the classical notion of Coxeter-Conway frieze pattern~\cite{Cox,CoCo}.
Similarly to the classical frieze patterns,
2-frieze patterns constitute a particular case of the
3-dimensional octahedron recurrence:
$$
T_{i + 1, j, k}\,T_{i -1, j, k}=
T_{i , j+ 1, k}\, T_{i , j- 1, k}
-T_{i , j, k+1}\,T_{i , j, k-1},
$$
which may be called the Dodgson condensation formula (1866)
and which is also known in the mathematical physics literature
as the discrete Hirota equation (1981).
More precisely, assume $T_{-1, j, k}=T_{2, j, k}=1$
and $T_{i , j, k}=0$ for $i\leq-2$ and $i\geq3$.
Then $T_{0, j, k}$ and $T_{1, j, k}$ form a 2-frieze.
More general recurrences called the $T$-systems and their relation to cluster
algebras were studied recently, see \cite{Hen,DK2,Kel} and references therein.
In particular, periodicity and positivity results, typical for cluster algebras,
were obtained.

The above 2-frieze rule was mentioned in  \cite{Pro} as a variation on the Coxeter-Conway frieze pattern. 
What we call a 2-frieze pattern also appeared in \cite{Ber} in a form of duality on $\SL_3$-tilings. 
To the best of our knowledge, 2-frieze patterns have not been studied in detail before.

We are particularly interested in 2-frieze patterns
bounded from above and from below by a row of $1$'s and two rows of $0$'s:
$$
 \begin{matrix}
 \cdots
& 0&0&0&0&0&\cdots
 \\[4pt]
\cdots& 0&0&0&0&0&\cdots
 \\[4pt]
\cdots& 1&1&1&1&1&\cdots
 \\[4pt]
\cdots&v_{0,0}&v_{\half,\half}&v_{1,1}&v_{\frac{3}{2},\frac{3}{2}}&v_{2,2}&\cdots
 \\
&\vdots &\vdots &\vdots &\vdots &\vdots &&\\
 \cdots
& 1&1&1&1&1&\cdots \\[4pt]
\cdots& 0&0&0&0&0&\cdots
 \\[4pt]
\cdots& 0&0&0&0&0&\cdots
\end{matrix}
$$
that we call \textit{closed} 2-frieze patterns.
We call the \textit{width} of a closed pattern  the number of rows
between the two rows of~1's.
In the sequel, we will often omit the rows of $0$'s in order to simplify the notation.

We introduce the following notation:
$$
\F_n=\left\{
\hbox{closed 2-friezes of width $n-4$}
\right\}
$$
for the space of all closed (complex or real) 2-frieze patterns.
Here and below the term ``space'' is used to identify a set of
objects that we wish to endow with a geometric structure
of (algebraic, smooth or analytic) variety.
We denote by $\F^0_n\subset\F_n$ the subspace of
closed friezes of width~$n-4$ such that all their entries
are \textit{real positive}.

Along with the octahedron recurrence,
the space of all 2-frieze patterns is closely related to
the theory of cluster algebras and cluster manifolds \cite{FZ1}.
In this paper, we explore this relation.

\subsection{Geometric version: moduli spaces of $n$-gons}\label{GeomSec}

An $n$-\textit{gon} in the projective plane is given by a cyclically ordered $n$-tuple of points
$\{v_1,\ldots,v_n\}$ in $\pP^2$ such that
no three consecutive points belong to the same projective line.
In particular,  $v_i\not=v_{i+1}$, and $v_i\not=v_{i+2}$.
However, one may have $v_i=v_j$, if $|i-j|\geq3$.
We understand the $n$-tuple $\{v_1,\ldots,v_n\}$ as an infinite cyclic sequence,
that is, we assume $v_{i+n}=v_i$, for all $i=1,\ldots,n$.

We denote  the space of all $n$-gons modulo projective
equivalence  by $\cC_n$:
$$
\cC_n=
\left\{
(v_1,\ldots,v_n)\in\pP^2
\left|\;
\det(v_i,v_{i+1},v_{i+2})\not=0,\;i=1,\ldots,n
\right.
\right\}
/\PSL_3.
$$
The space $\cC_n$ is a $(2n-8)$-dimensional algebraic variety.

Similarly, one defines an $n$-gon in 3-dimensional
vector space (over $\R$ or $\C$): this is a cyclically ordered $n$-tuple 
of vectors $\{V_1,\ldots,V_n\}$ satisfying the unit determinant condition
$$
\det(V_{i-1},V_i,V_{i+1})=1
$$
for all indices $i$ (understood cyclically). 
The group $\SL_3$ naturally acts on $n$-gons. 
The space of equivalence classes is denoted by $\tilde \cC_n$; 
this is also a $(2n-8)$-dimensional algebraic variety.

Projectivization gives a natural map  $\tilde \cC_n\to \cC_n$. 
It is shown in \cite{OST} that this map is bijective if $n$ is not divisible by 3; 
see Section \ref{ClDiGo}.

We show that  the space of closed 2-frieze patterns $\F_n$ is isomorphic to
the space of polygons~$\tilde \cC_n$.
This also means that $\F_n$ is isomorphic to $\cC_n$,
provided $n$ is not a multiple of 3.

An $n$-gon $\{V_1,\ldots,V_n\}$ 
in $\R^3$ is called {\it convex} if, for each $i$, all the vertices $V_j,\ j\neq i-1,i$, 
lie on the positive side of the plane generated by $V_{i-1}$ and $V_{i}$,  
that is, $\det (V_{i-1},V_{i}, V_j)>0$.
See Figure \ref{FirstFig}.  Let $\tilde \cC_n^0\subset \tilde \cC_n$ denote
the space of convex $n$-gons in $\R^3$. 
We show that the space of positive real 2-friezes $\F^0_n$ is isomorphic to
the space of convex polygons $\tilde \cC_n^0$.

\begin{figure}[hbtp]
\includegraphics[width=6cm]{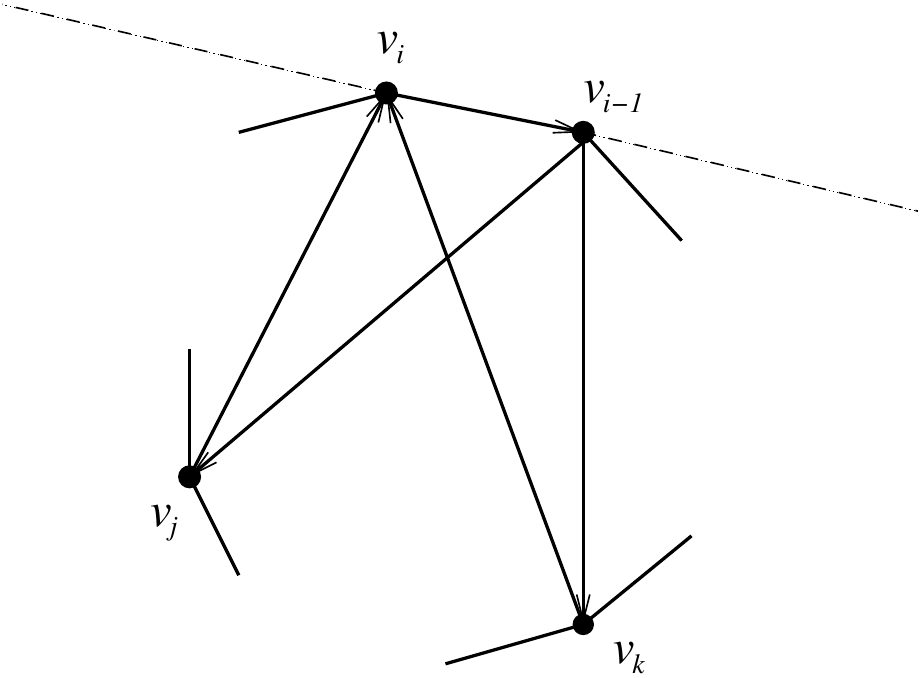}
\caption{A convex polygon.}
\label{FirstFig}
\end{figure}

\begin{rem}
\label{TwistRem}
A more general space of \textit{twisted} $n$-gons in $\pP^2$ (and similarly in 3-dimensional vector space)
was considered in \cite{Sch,OST}.
A twisted $n$-gon in $\pP^2$ is a map $\varphi:\Z\to\pP^2$ such that
no three consecutive points, $\varphi(i), \varphi(i+1), \varphi(i+2)$, belong to the same projective line and
$$
\varphi(i+n)=M(\varphi(i)),
$$
where $M\in\PSL_3$ is a fixed element, called the \textit{monodromy}.
If the monodromy is trivial, $M=\Id$, then the twisted $n$-gon is an $n$-gon in the above sense.
In \cite{Sch,OST} two different systems of coordinates were introduced 
and used to study the space of twisted $n$-gons and the transformation under the pentagram map. 
\end{rem}

\subsection{Analytic version: the space of difference equations} \label{Analvers}
Consider a difference equation of the form
\begin{equation} 
\label{recur}
V_{i}=a_i \,V_{i-1}-b_i\,V_{i-2} + V_{i-3},
\end{equation}
where $a_i,b_i\in\C$ or $\R$ are $n$-periodic:
$a_{i+n}=a_i$ and $b_{i+n}=b_i$, for all $i$.
A solution $V=(V_i)$ is a sequence of numbers $V_i\in\C$ or $\R$ satisfying (\ref{recur}).
The space of solutions of (\ref{recur}) is 3-dimensional.
Choosing three independent solutions,
we can think of~$V_i$ as vectors in $\C^3$ (or $\R^3$).
The $n$-periodicity of $(a_i)$ and~$(b_i)$
then implies that there exists a matrix $M\in\SL_3$
called the \textit{monodromy matrix}, such that 
$$
V_{i+n}=M\left(V_i\right).
$$

The space of all the equations \eqref{recur} is nothing other than the vector space $\C^{2n}$ (or $\R^{2n}$, in the real case), since
$(a_i,b_j)$ are arbitrary numbers.
The space of equations with trivial monodromy,
$M=\Id$, is an algebraic manifold of dimension $2n-8$,
since the condition $M=\Id$ gives eight polynomial equations (of degree $n-3$).

We show that  the space of closed 2-frieze patterns $\F_n$ is isomorphic to
the space of equations~\eqref{recur} with trivial monodromy.

\subsection{The pentagram map and cluster structure}

The pentagram map, $T$, see Figure \ref{penta}, was initially defined by R. Schwartz \cite{Sch1}
on the space of (convex) closed $n$-gons in $\RP^2$.
This map associates to an $n$-gon another $n$-gon formed by segments of
the shortest diagonals.
Since~$T$ commutes with the $\SL(3,\R)$-action, it is well-defined on the quotient space $\cC_n$.
Complete integrability of~$T$ on $\cC_n$ was conjectured and partially established in \cite{Sch}.

\begin{figure}[hbtp]
\includegraphics[width=9cm]{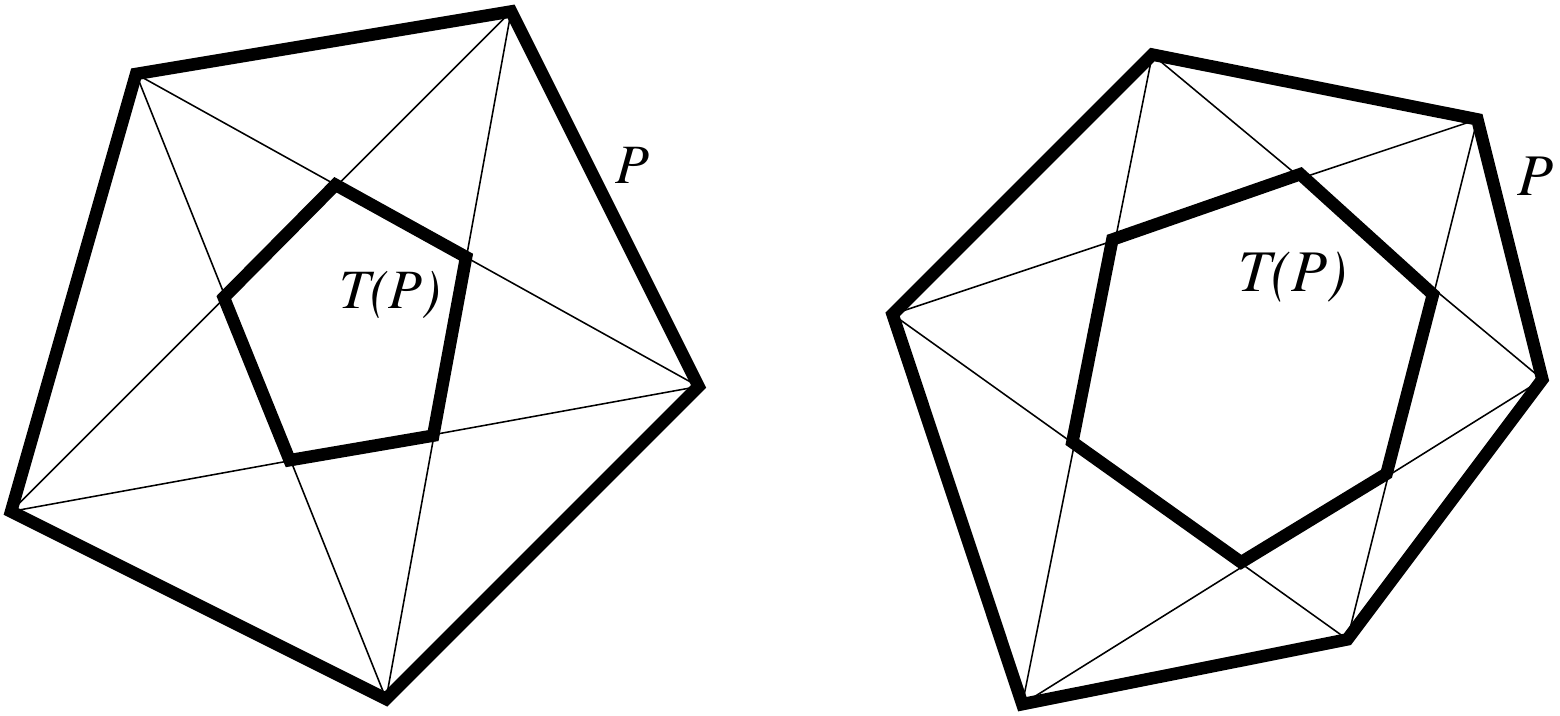}
\caption{The pentagram map.}
\label{penta}
\end{figure}

The integrability results on the pentagram map were originally established in the case of
the space of twisted $n$-gons (see Remark \ref{TwistRem}).
Complete integrability of $T$ on this space was proved in \cite{OST} and the relation to
cluster algebras was noticed.
Explicit formulas for iterated pentagram map $T^k$ were recently found \cite{Gli} using an alternative system 
of parametrizations of the twisted $n$-gons. 
These formulas involve the theory of cluster algebras and $Y$-patterns.
Here, we describe a structure of cluster manifold on the space of closed $n$-gons, which is a different question. The relation of our approach with the one by Glick deserves a thorough study; we plan to
consider this question in near future. It is not clear, at the time of writing, how the cluster structure on  the space $\cC_n$ that we describe in this paper is related to complete integrability of the map~$T$ on $\cC_n$ proved in \cite{Sol} and \cite{OST2}.

\section{Definitions and main results}

\subsection{Algebraic and numerical friezes}\label{ANuF}

It is important to distinguish the \textit{algebraic} 2-frieze patterns,
where the entries are algebraic functions,
and the \textit{numerical} ones where the entries are real numbers.

Our starting point is the algebraic frieze bounded from above by a row of 1's
(we also assume that there are two rows of 0's above the first row of 1's).
We denote by $A_i,B_i$ the entries in the first non-trivial row:
 \begin{equation}
 \label{FriezeFAB}
 \begin{matrix}
 \cdots
& 1&1&1&1&1&\cdots
 \\[4pt]
\cdots&B_i&A_i&B_{i+1}&A_{i+1}&B_{i+2}&\cdots
 \\
&\vdots &\vdots &\vdots &\vdots &\vdots &&
\end{matrix}
\end{equation}
The entries $A_i,B_i$ are considered as formal free variables.

\begin{prop}
\label{AlgProp}
The first two non-zero rows of \eqref{FriezeFAB} uniquely define an unbounded
(from below, left and right) 2-frieze pattern.
Every entry of this pattern is a polynomial in $A_i,B_i$.
\end{prop}

\noindent
This statement will be proved in Section \ref{XYZ}.

We denote the defined 2-frieze by $F(A_i,B_i)$.

Given a sequence of real numbers $(a_i,b_i)_{i\in\Z}$,
we define a \textit{numerical 2-frieze pattern} $F(a_i,b_i)$ as the evaluation
$$
F(a_i,b_i)=F(A_i,B_i)\big\vert_{A_i=a_i,\,B_i=b_i}.
$$
Note that one can often recover the whole numerical frieze $F(a_i,b_i)$ directly from
the two first rows (of 1's and $(a_i,b_i)$) by applying the pattern rule
but this is not always the case.
For instance this is not the case if there are too many zeroes among $\{a_i,b_i\}$.
In other words, there exist numerical friezes that are not evaluations
of $F(A_i,B_i)$.

\begin{ex}
The following 2-frieze pattern:
$$
 \begin{matrix}
 \cdots
& 1&1&1&1&1&\cdots
 \\
\cdots&0&0&0&0&0&\cdots
\\
\cdots&0&0&0&0&0&\cdots
\\
\cdots&0&0&0&0&0&\cdots
\\
 \cdots
& 1&1&1&1&1&\cdots
\end{matrix}
$$
is not an evaluation of some $F(A_i,B_i)$.
Indeed, if $a_i=b_i=0$ for all $i\in\Z$, then the 4-th row in $F(a_i,b_i)$
has to be a row of 1's.
This follows from formula \eqref{FirstVal} below.
\end{ex}

The above example is not what we called a numerical 2-frieze pattern
and we will not consider such friezes in the sequel.
We will restrict our considerations to evaluations of $F(A_i,B_i)$.

\subsection{Closed frieze patterns}

A numerical 2-frieze pattern $F(a_i,b_i)$ is closed if 
it contains a row of 1's followed by two rows of zeroes:
\begin{equation}
\label{DefClo}
\begin{matrix}
 \cdots
& 1&1&1&1&1&\cdots
 \\[4pt]
\cdots&b_i&a_i&b_{i+1}&a_{i+1}&b_{i+2}&\cdots
 \\
&\vdots &\vdots &\vdots &\vdots &\vdots &&
\\
 \cdots
& 1&1&1&1&1&\cdots
\\
 \cdots
& 0&0&0&0&0&\cdots
\\
 \cdots
& 0&0&0&0&0&\cdots
 \\
&\vdots &\vdots &\vdots &\vdots &\vdots &&
\end{matrix}
\end{equation}

The following statement is proved in Section \ref{NumSec}.

\begin{prop}
\label{Period}
A closed $2$-frieze pattern of width $(n-4)$ has the following properties.

(i)
It is $2n$-periodic in each row, i.e., $v_{i+n,j+n}=v_{i,j}$,
in particular $a_{i+n}=a_i$ and $b_{i+n}=b_i$.

(ii)
It is $n$-periodic in each diagonal,
i.e., $v_{i+n,j}=v_{i,j}$ and $v_{i,j+n}=v_{i,j}$.

(iii)
It satisfies the following additional glide symmetry:
$v_{i,j}=v_{j+n-\frac{5}{2},\,i+\frac{5}{2}}$.
\end{prop}
\noindent
The statement of part (iii) means that, after $n$ steps,
the pattern is reversed with respect to the horizontal symmetry axis.

As a consequence of Proposition \ref{Period},
a closed 2-frieze of width $n-4$ consists of
 periodic blocks of size $(n-2)\times2n$.
Taking into account the symmetry of part (iii), the 2-frieze
is determined by a fragment of size
$(n-2)\times{}n$.

\begin{ex}
\label{FirstEx}
(a)
The following fragment completely determines a closed 2-frieze
of width 2.
$$
\begin{array}{rrrrrrr|rrrrrrrr}
\ldots&1&1&1&1&1&1&1&1&1&1&1&1&\ldots\\
\ldots&1&1&4&6&2&1  & 2&3&2&2&4&3&\ldots \\
\ldots&2&3&2&2&4&3 &1&1&4&6&2&1&\ldots\\
\ldots&1&1&1&1&1&1&1&1&1&1&1&1&\ldots
\end{array}
$$
The additional symmetry from Proposition~\ref{Period}, part (iii),
switches the rows every 6 steps.

(b)
The following integral numerical 2-frieze pattern
$$
\begin{array}{rrrrrrr|rrrrrrrr}
\ldots&1&1&1&1&1&1&1&1&1&1&1&1&\ldots\\
\ldots&1&3&5&2&1&3& 5&2&1&3&5&2&\ldots\\
\ldots&5&2&1&3&5&2 &1&3&5&2&1&3&\ldots\\
\ldots&1&1&1&1&1&1&1&1&1&1&1&1&\ldots
\end{array}
$$
is closed of width 2.
This corresponds to $n=6$ so that this 2-frieze pattern is understood as 12-periodic 
(and not as 4-periodic!).
\end{ex}

\subsection{Closed 2-friezes, difference equations and $n$-gons} \label{ClDiGo}

Consider an arbitrary numerical 2-frieze pattern $F(a_i,b_i)$.
By Proposition \ref{Period}, a necessary condition of closeness is:
$$
a_{i+n}=a_i,
\qquad
b_{i+n}=b_i
$$
that we assume from now on.
We then say that $F(a_i,b_i)$ is $2n$-\textit{periodic}.
 
Associate to $F(a_i,b_i)$ the difference equation \eqref{recur}.
The first main result of this paper is the following criterion of closeness.
The statement is very similar to a result of \cite{CoCo}.

\begin{thm}
\label{MainOne}
A $2n$-periodic 2-frieze pattern $F(a_i,b_i)$ is closed if and only if
the monodromy of the corresponding difference equation \eqref{recur} is trivial:
$
M=\Id.
$
\end{thm}
\noindent
This theorem will be proved in Section \ref{NumSec}.

The variety $\F_n$
of closed 2-frieze patterns \eqref{DefClo} is thus identified with the space
of difference equations (\ref{recur}) with trivial monodromy.
The latter space was considered in \cite{OST}.
In particular, the following geometric realization holds.

\begin{prop}
\label{OSTProp}
\cite{OST}
The space of difference equations (\ref{recur}) 
with trivial monodromy is isomorphic to the space of $\SL_3$-equivalence classes of polygons 
$\tilde \cC_n$ in $3$-space.
If $n$ is not divisible by~$3$, then this space is also isomorphic to the space $\cC_n$ of 
projective equivalence classes of polygons in the projective plane.
\end{prop}

\noindent
For completeness, we give a proof   in Section \ref{NumSec}.

It follows from Theorem \ref{MainOne} that the variety $\F_n$ is isomorphic to
$\tilde \cC_n$, and also to $\cC_n$, provided~$n$ is not a multiple of 3.
In order to illustrate the usefullness of this isomorphism, in Section \ref{ConvSubsect},
we prove the following statement.

\begin{prop}
\label{PosConvPro}
All the entries of a real $2$-frieze pattern are positive
if and only if the corresponding $n$-gon in $\R^3$ is convex.
\end{prop}

We understand convex $n$-gons in $\RP^2$ as polygons that lie in an affine chart and are convex therein.
If $n$ is not a multiple of 3, then convexity of an $n$-gon in $\R^3$ is equivalent to convexity of its projection to $\RP^2$, see Section \ref{ConvSubsect}.
In Section \ref{FGSec}, we show that the space of convex $3m$-gons is isomorphic to
the space of pairs of $2m$-gons inscribed one into the other.
This space was studied by Fock and Goncharov \cite{FG, FG1}.

\subsection{Cluster structure}
The theory of cluster algebras introduced and developed by Fomin and Zelevinsky
\cite{FZ1}-\cite{FZ4}
is a powerful tool for the study of many classes of algebraic varieties.
This technique is crucial for the present paper.
Note that the relation of octahedron recurrence and $T$-systems to cluster algebras
is well-known, see, e.g., \cite{Volk,DiF,DK1,DK2,Kel}.
Some of our statements are very particular cases of known results
and are given here in a more elementary way for the sake of completeness.

It was first proved in~\cite{CaCh} that the space
of the classical Coxeter-Conway friezes has a cluster structure
related to the simplest Dynkin quiver $A_n$
(see also Appendix).
In Section \ref{ClustSect}, we prove a similar result.

Consider  the following oriented graph (or quiver) that we denote by $\Qc$:
\begin{equation}
\label{Graph}
\xymatrix{
&1\ar@{->}[r]
&2\ar@{<-}[r]\ar@{->}[d]
&3\ar@{->}[r]
&\cdots
&\cdots \ar@{<-}[r]
&n-5\ar@{->}[r]
&n-4\ar@{->}[d]\\
&
n-3\ar@{<-}[r]\ar@{->}[u]
&n-2\ar@{->}[r]
&n-1\ar@{->}[u]\ar@{<-}[r]
&\cdots&\cdots \ar@{->}[r]
&2n-9\ar@{->}[u]\ar@{<-}[r]
&2n-8\\
}
\end{equation}
if $n$ is even, or with the opposite orientation of the last square if $n$ is odd
(note that in the case $n=5$ the graph consists only in two vertices linked by one arrow).
The graph~$\Qc$ is the product of two Dynkin quivers: $\Qc=A_2*A_{n-4}$.

\begin{ex}
The graph (\ref{Graph}) is a particular case of the graphs
related to the cluster structure on Grassmannians, see \cite{Sco}.
The cluster algebra considered in Section \ref{ClustSect}
can be viewed as a specialization of the one considered by Scott.
In particular, for $n=5$, this is nothing else but the simplest Dynkin graph~$A_2$.
For $n=6,7,8$, the graph $\Qc$ is equivalent (by a series of mutations) to $D_4 ,E_6,E_8$, respectively.
The graph $\Qc$ is of infinite type for $n\geq9$.
For $n=9$, the graph $\Qc$ is equivalent to the infinite-type graph $E_8^{1,1}$.
The relation of our approach with the one by Scott deserves a thorough study.
\end{ex}

The cluster algebra associated to $\Qc$ is of infinite type for $n\geq9$.
In Section \ref{FunClo}, we define a finite subset $\cZ$ in the set of all clusters
associated to $\Qc$.
More precisely, the subset $\cZ$ is the set of all clusters that can be obtained from
the initial cluster by series of mutations at vertices that do not  belong to two 3-cycles.

\begin{thm}
\label{MainSecond}
(i)
The cluster coordinates 
$(x_1,\ldots,x_{n-4},y_1,\ldots,y_{n-4})_\zeta$, where $\zeta \in \cZ$,
define a biregular isomorphism between $(\C^*)^{2n-8}$ 
and a Zariski open subset of $\F_n$.

(ii)
The coordinates $(x_1,\ldots,x_{n-4},y_1,\ldots,y_{n-4})_\zeta$ 
restrict to a bijection between $\R_{>0}^{2n-8}$ and~$\F_n^0$.
\end{thm}

\noindent
Theorem \ref{MainSecond} provides good coordinate systems for the study of the space $\cC_n$,
different from the known coordinate systems described in \cite{Sch,OST}.

We think that the ``restricted'' set of cluster coordinates $\cZ$ is an interesting object
that perhaps has a more general meaning, this question remains open.
 In Section \ref{NoLabel}, we define the so-called smooth cluster-type manifold corresponding to the subset $\cZ$.
 The space of all 2-friezes $\F_n$ is precisely this manifold completed by some singular points.

The following proposition is a typical statement that can be
obtained using the cluster structure. 
A double zig-zag is a choice of two adjacent entries in each row of a pattern 
so that the pair of entries in each next row is directly underneath 
the pair above it, or is offset of either one position right or left, 
see Section \ref{FunClo} below for details. 

\begin{prop}
\label{Boundme}
Every frame bounded from the left by a double zig-zag of $1$'s:
$$
 \begin{matrix}
1& 1&1&1&1&1&\cdots
 \\[4pt]
&1&1&\cdots&&&
 \\[4pt]
&&1&1&\cdots&&
 \\[4pt]
 &1&1&\cdots&&&
 \\
&\vdots &\vdots & & & &&\\
& 1&1&1&1&1&\cdots
\end{matrix}
$$
can be completed (in a unique way) to a closed $2$-frieze pattern with positive integer
entries.
\end{prop}
This statement will be proved in Section \ref{FunClo}.
This is a direct generalization of a Coxeter-Conway result \cite{CoCo}.

\subsection{Arithmetic 2-friezes}

Let us consider closed 2-frieze patterns of period $2n$ consisting of positive integers
(like in Example \ref{FirstEx}); we call such 2-friezes {\it arithmetic}.
The classification of such patterns is a fascinating problem formulated in \cite{Pro}.
This problem remains open.

In Section \ref{Trick}, we present an inductive method of constructing
a large number of arithmetic 2-frieze patterns.
This is a step towards  the classification.

Consider two closed arithmetic 2-frieze patterns, $F(a_i,b_i)$ and $F(a'_i,b'_i)$,
one of them of period~$2n$ and the other one of period $2k$, with coefficients
$$
b_1,\;a_1,\;b_2,\;a_2, \ldots, b_n,\;a_n
\qquad
b'_1,\;a'_1,\;b'_2,\;a'_2, \ldots, b'_k,\;a'_k,
$$
respectively.
We call the \textit{connected summation} the following way to glue them together
and obtain a  2-frieze pattern of period $2(n+k-3)$.

\begin{enumerate}
\item
Cut the first one  at an arbitrary place, say between $b_2$ and $a_2$.
\item
Insert $2(k-3)$ integers:
$a'_2,\;b'_3,\ldots,a'_{k-2},\;b'_{k-1}$.
\item
Replace the three left and the three right neighbouring entries by:
\begin{equation}
\label{MultStabRow}
\begin{array}{rclll}
\textstyle
\left(
b_1,\;a_1,\;b_2
\right)
&\to&
(b_1+b'_1,& a_1+a'_1+b_2\,b'_1,& 
b_2+b'_2)
\\[6pt]
\textstyle
(a_2,\;b_3,\;a_3)
&\to&
(a_2+a'_{k-1},& b_3+b'_k+a_2\,a'_k,& a_3+a'_k),
\end{array}
\end{equation}
leaving the other $2(n-3)$ entries $b_4,a_4, \ldots, b_n,a_n$ unchanged.
\end{enumerate}

In Section \ref{Trick},
we will prove the following statement.

\begin{thm}
\label{MultiStabProp} Connected summation yields a closed $2$-frieze pattern of period $2(n+k-3)$.
If $F(a_i,b_i)$ and $F(a'_i,b'_i)$ are closed arithmetic $2$-frieze patterns, then
their connected sum is also a closed arithmetic $2$-frieze pattern.
\end{thm}

\noindent
In Sections \ref{StabOneSec} and \ref{StabTwoSec},
we explain the details in the first non-trivial cases: $k=4$ and $5$,
that we call ``stabilization''.

The classical Coxeter-Conway integral frieze patterns were classified in \cite{CoCo}
with the help of a similar stabilization procedure.
In particular, a beautiful relation with triangulations of an $n$-gon
(and thus with the Catalan numbers) was found making the result more
attractive. 
Unfortunately, the above procedure of connected summation 
does not lead to  classification of arithmetic 2-frieze patterns.
This is due to the fact that, unlike the Coxeter-Conway integral frieze patterns,
not every integral 2-frieze pattern is a connected sum of smaller ones,
see examples below.

\section{Algebraic 2-friezes}\label{XYZ}

The goal of this section, is to describe various ways to calculate
the frieze \eqref{FriezeFAB}.
This will imply Proposition \ref{AlgProp}.

\subsection{The pattern rule}

Recall that we denote by $(v_{i,j})_{(i,j)\in \Z^2}$
and $(v_{i+\half,j+\half})_{(i,j)\in \Z^2}$ the entries of the frieze organized as follows
$$
 \xymatrix{
 &\cdots\ar@{--}[rd]
&1\ar@{-}[rd]\ar@{-}[ld]
&1\ar@{--}[ld]\ar@{--}[rd]
&1\ar@<2pt>@{-}[ld]\ar@{-}[rd]
&\cdots \ar@{--}[ld]\\
&\cdots \ar@<2pt>@{-}[rd]
&v_{i-\half,i-\half}\ar@{--}[rd]\ar@{--}[ld]
& v_{i,i}\ar@{-}[ld]\ar@{-}[rd]
&v_{i+\half,i+\half}\ar@{--}[rd]\ar@{--}[ld]
&\cdots \ar@{-}[ld]\\ 
&\cdots  \ar@{--}[rd]
& v_{i,i-1}\ar@<2pt>@{-}[rd]\ar@<2pt>@{-}[ld]
&v_{i+\half,i-\half}\ar@{--}[ld]\ar@{--}[rd]
& v_{i+1,i}\ar@{-}[ld]\ar@{-}[rd] 
&\ar@{--}[ld] \cdots\\
&\cdots 
&\cdots &\cdots &\cdots &\cdots  
}
$$
In the algebraic frieze, we assume:
$$
v_{i,i}=A_i
\qquad
v_{i-\half,i-\half}=B_i,
$$
see \eqref{FriezeFAB}.

The first way to calculate the entries in the frieze \eqref{FriezeFAB}
is a direct inductive application of the pattern rule.

\subsection{The determinant formula}

The most general formula for the elements of the pattern is the following
determinant formula generalizing that of Coxeter-Conway \cite{CoCo}.

\begin{prop}
\label{TabProp}
One has
\begin{equation}
\label{DetTab}
v_{i,j}=
\left|
\begin{array}{llllll}
A_{j}&B_{j+1}&1&&&\\
1&A_{j+1}&B_{j+2}&1&&\\
&\;\ddots&\; \ddots& \;\ddots&\; \ddots&\\
&&1&A_{i-2}&B_{i-1}&1\\
&&&1&A_{i-1}&B_{i}\\
&&&&1&A_i
\end{array}
\right|
\end{equation}
for $i\geq{}j\in\Z$.
The element $v_{i+\half,j+\half}$ is obtained from $v_{i,j}$
by replacing $(A_k,B_{k+1})\to(B_{k+1},A_{k+1})$.
\end{prop}

\begin{proof}
The 2-frieze pattern rule for $v_{i-1/2,j+1/2}$ reads
$$
\textstyle
v_{i-1,j}\,v_{i,j+1}=v_{i-1,j+1}\,v_{i,j}+v_{i-\half,j+\half}.
$$
Using induction on $i-j$, understood as the row number,
we assume that  the formula for $v_{k,\ell}$ holds for $k-\ell<i-j$,
so that all the terms of the above equality except $v_{i,j}$ are known.
The result then follows from the Dodgson formula.
\end{proof}

The algebraic frieze looks as follows:
\begin{equation}
\label{FirstVal}
\begin{array}{lccccc}
1&1&1& 1&1& \cdots
 \\[10pt]
\cdots\quad&A_0&B_1&A_1&B_2&\cdots
 \\[12pt]
&\cdots &A_0A_1-B_1&B_1B_2-A_1&A_1A_2-B_2&\cdots
 \\[12pt]
& &\;\;\; \ldots& A_0A_1A_2-A_2B_1&B_1B_2B_3-A_1B_3 &\ldots
 \\
&&&\;-A_0B_2+1&\;-A_2B_1+1&
\\[12pt]
&&&&A_0A_1A_2A_3-A_2A_3B_1&\\
&&&\;\;\;\;\; \cdots&\;-A_0A_3B_2-A_0A_1B_3&\ldots\\
&&&&\;+B_1B_3+A_0+A_3&
\end{array}
\end{equation}

\subsection{Recurrence relations on the diagonals}

Let us introduce the following notation.
\begin{enumerate}
\item
The diagonals pointing ``North-East'' (that contain all the elements $v_{i,.}$ with $i$ fixed);
are denoted by $\D_i$.
\item
The diagonals pointing ``South-East'' (that contain all the elements  $v_{.,j}$ with $j$ fixed) 
are denoted by $\Db_j$.
\item
The (horizontal) rows $\{v_{i,j}\;\vert\;i-j=\mathrm{const}\}$
are denoted by $R_{i-j}$.
\end{enumerate}
$$
\xymatrix{
&\Db_{j-1}\ar@{->}[rrrrdddd]
 &\Db_{j-\half}\ar@{-->}[rrrrdddd]
&\Db_j\ar@{->}[rrrrdddd]
&
&\D_i\ar@{<-}[lllldddd]
&\D_{i+\half}\ar@{<--}[lllldddd]
 &\D_{i+1}\ar@{<-}[lllldddd]
&
 &
 &
 \\
R_{i-j}\ar@<10pt>@{..>}[rrrrrrrr]&&&&&&&&&&&&
 \\
R_{i-j+1}\ar@<8pt>@{..>}[rrrrrrrr]&&&&&&&&&&&&&\\
&&&&&&&&&&&&&&&\\
&&&&&&&&&&&&&&&&&&&&&&&&
}
\label{LinesFrieze}
$$

\begin{prop}\label{proprec}
One has the following recurrence relations.
For all $i\in \Z$,
\begin{equation}
\label{rec1}
\begin{array}{rcl}
\D_{i}&=&A_i\D_{i-1}-B_i\D_{i-2}+\D_{i-3},\\[6pt]
\Db_{i}&=&A_i\Db_{i+1}-B_{i+1}\Db_{i+2}+\Db_{i+3},
\end{array}
\end{equation}
 and 
\begin{equation}
\label{rec2}
\begin{array}{rcl}
\D_{i+\half}&=&B_{i+1}\D_{i-\half}-A_i\D_{i-\frac{3}{2}}+\D_{i-\frac{5}{2}},\\[6pt]
\Db_{i+\half}&=&B_{i+1}\Db_{i+\frac{3}{2}}-A_{i+1}\Db_{i+\frac{5}{2}}+\Db_{i+\frac{7}{2}}.
\end{array}
\end{equation}
\end{prop}

\begin{proof}
Straighforward using the determinant formula \eqref{DetTab}.
\end{proof}

Note also that the difference equations \eqref{rec1} and \eqref{rec2} are \textit{dual} to each other,
see \cite{OST}.

\subsection{Relation to $\SL_3$-tilings}

An $\SL_k$-tiling is an infinite matrix such that any principal $k\times k$-minor
(i.e., a minor with contiguous row and column indices)
is equal to $1$. These  $\SL_k$-tilings were introduced and studied in \cite{Ber}. 
Following \cite{Ber}, we note that
an algebraic 2-frieze pattern contains two $\SL_3$-tilings. 
\begin{prop}
\label{LinComb}
The subpatterns
$(v_{i,j})_{i,j\in\Z}$ and
$(v_{i,j})_{i,j\in\Z+\half}$ of $F(A_i,B_i)$
are both $\SL_3$-tilings.
\end{prop}

\begin{proof}
Using Dodgson's formula, one obtains
\begin{equation*}
\begin{array}{lll}
&
\left|
\begin{array}{lll}
v_{i-1,j-1}&v_{i-1,j}&v_{i-1,j+1}\\
v_{i,j-1}&v_{i,j}&v_{i,j+1}\\
v_{i+1,j-1}&v_{i+1,j}&v_{i+1,j+1}\\
\end{array}
\right| 
v_{i,j}
\\[20pt]
&=\left|
\begin{array}{lll}
v_{i-1,j-1}&v_{i-1,j}\\
v_{i,j-1}&v_{i,j}
\end{array}
\right|
\left|
\begin{array}{lll}
v_{i,j}&v_{i,j+1}\\
v_{i+1,j}&v_{i+1,j+1}
\end{array}
\right|
-
\left|
\begin{array}{lll}
v_{i-1,j}&v_{i-1,j+1}\\
v_{i,j}&v_{i,j+1}
\end{array}
\right|
\left|
\begin{array}{lll}
v_{i,j-1}&v_{i,j}\\
v_{i+1,j-1}&v_{i+1,j}\\
\end{array}
\right|\\[16pt]
&=v_{i-\half,j-\half}\,v_{i+\half,j+\half}-
v_{i-\half,j+\half}\,v_{i+\half,j-\half}\\[10pt]
&= v_{i,j}.
\end{array}
\end{equation*}
If follows from Proposition \ref{TabProp}, that $v_{i+1,j-1}\not=0$.
One obtains
$$
\left|
\begin{array}{lll}
v_{i,j-2}&v_{i,j-1}&v_{i,j}\\[4pt]
v_{i+1,j-2}&v_{i+1,j-1}&v_{i+1,j}\\[4pt]
v_{i+2,j-2}&v_{i+2,j-1}&v_{i+2,j}\\[4pt]
\end{array}
\right|
=1.
$$
Hence the result.
\end{proof}

The two $\SL_3$-tilings are dual to each other in the sense of \cite{Ber}. The converse statement also holds: one can construct a 2-frieze pattern from an $\SL_3$-tiling by superimposing the tiling on its dual, see \cite{Ber}.

\section{Numerical friezes}\label{NumSec}

In this section,
we prove Propositions \ref{Period}, \ref{OSTProp}, \ref{PosConvPro} and Theorem \ref{MainOne}. We also discuss a relation with a moduli space of Fock-Goncharov \cite{FG, FG1}.

\subsection{Entries of a numerical frieze}

Consider a numerical 2-frieze $F(a_i,b_i)$.
Its entries $v_{i,j}$ can be expressed as determinants involving
solutions of the corresponding difference equation \eqref{recur}, understood as vectors in 3-space.

\begin{lem}
\label{DetLem}
One has
\begin{equation}
\label{MalDetEq}
v_{i,j}=
\left|
V_{j-3},\,V_{j-2},\,V_{i}
\right|,
\qquad
v_{i-\half,j-\half}=
\left|
V_{i-1},\,V_{i},\,V_{j-3}
\right|,
\end{equation}
where $V=(V_i)$ is any solution of \eqref{recur} such that
$\left|
V_{i-2},\,V_{i-1},\,V_{i}
\right|=1$.
\end{lem}

\begin{proof}
Consider the diagonal $\Db_j$ of the frieze,
its elements $v_{i,j}$ are labeled by one index~$i\in\Z$.
We proceed by induction on $i$.

The base of induction is given by the three trivial elements
$v_{j-3,j}=v_{j-2,j}=0$ and $v_{j-1,j}=1$ which obviously satisfy \eqref{MalDetEq}.

The induction step is as follows.
According to formula \eqref{rec1}, the elements $v_{i,j}$ satisfy
the recurrence \eqref{recur}.
One then has
$$
\begin{array}{rcl}
v_{i,j}&=&a_i\,v_{i-1,j}-b_i\,v_{i-2,j}+v_{i-3,j}\\[6pt]
&=&a_i\left|
V_{j-3},\,V_{j-2},\,V_{i-1}
\right|-b_i\left|
V_{j-3},\,V_{j-2},\,V_{i-2}
\right|+\left|
V_{j-3},\,V_{j-2},\,V_{i-3}
\right|\\[6pt]
&=&\left|
V_{j-3},\,V_{j-2},\,a_iV_{i-1}-b_iV_{i-2}+V_{i-3}
\right|\\[6pt]
&=&\left|
V_{j-3},\,V_{j-2},\,V_{i}.
\right|
\end{array}
$$
Hence the result.

The proof in the half-integer case is similar.
\end{proof}

\subsection{Proof of Proposition \ref{Period}}

Consider a numerical frieze $F(a_i,b_i)$
and assume that this frieze is closed, as in \eqref{DefClo},
of width $n-4$.
Choosing the diagonal $\Db_i$, let us determine the last
non-trivial element $v_{i+n-5,i}$:
$$
\begin{matrix}
1&1&1&1&1&1&1
 \\[4pt]
&a_i&&&&&
 \\
& &\ddots & & &&&
 \\[4pt]
 &&&{\bf v_{i+n-5,i}}&&&
\\[6pt]
 1&1&{\bf 1}&1&{\bf 1}&1&1
\\
0&{\bf 0}&0&0&0&{\bf 0}&0
\\
{\bf 0}&0&0&0&0&0&{\bf 0}
\end{matrix}
$$
Using the first recurrence relation in \eqref{rec1}, one has
$$
v_{i+n-2,i}=a_{i+n-2}\,v_{i+n-3,i}-b_{i+n-2}\,v_{i+n-4,i}+v_{i+n-5,i}.
$$
This implies
$v_{i+n-5,i}=b_{i+n-2}$, since $v_{i+n-2,i}=v_{i+n-3,i}=0$ and $v_{i+n-4,i}=1$.
On the other hand, using the second recurrence relation in \eqref{rec1}, one has
$$
v_{i+n-5,i-3}=a_{i-3}\,v_{i+n-5,i-2}-b_{i-2}\,v_{i+n-5,i-1}+v_{i+n-5,i}.
$$
This implies
$v_{i+n-5,i}=b_{i-2}.$
Combining these two equalities, one has $n$-periodicity:
$$
b_{i+n-2}=b_{i-2}.
$$
Similarly, choosing the diagonal $\Db_{i+\half}$, one obtains
$n$-periodicity of the $a_i$.
Finally, $2n$-periodicity on the first two rows implies $2n$-periodicity on each row.
Part (i) is proved.

In order to prove Part (ii), we continue to determine the entries of $\Db_i$:
$$
\begin{array}{rccccl}
v_{i+n-5,i}&&&&&
\\[6pt]
&1&&&&
\\
&&0&&&
\\
&&&0&&
\\
&&&&1&
\\[4pt]
&&&&&v_{i+n,i}
\end{array}
$$
Using \eqref{rec1}, we deduce that $v_{i+n-1,i}=1$ and,
using this relation again, 
$$
v_{i+n,i}=a_{i+n}=a_i=v_{i,i},
\qquad i\in\Z
$$
and similarly for $i\in\Z+\half$.
Part (ii) is proved.

Part (iii) follows from the equalities
$v_{i+n-5,i}=b_{i+n-2}=b_{i-2}$
proved in Part~(i).
Indeed, in the first non-trivial row, $b_{i-2}=v_{i-\frac{5}{2},i-\frac{5}{2}}$,
so that we rewrite the above equality as follows:
$v_{i,i}=v_{i+n-\frac{5}{2},i+\frac{5}{2}}$.
This means that the  first non-trivial row is related to the last one
by the desired glide symmetry.
Then using the 2-frieze rule we deduce that the same glide symmetry relates
the second row  with the one before the last, etc.
$$
\xymatrix{
1\ar@{--}[rd]\ar@<14pt>@{<->}^n[rrrrr]
&1\ar@{-}[ld]
&&\cdots&
 &1\ar@{--}[rd]\ar@<14pt>@{<-}^n[rr]
&1\ar@{-}[ld]\ar@{-}[rd]
&\cdots \ar@{--}[ld]
\\
 v_{\half,\half} \ar@{-}[rd] 
&v_{1,1}\ar@{--}[ld]
 \ar@{->}@/^/[rrrrdddd]
&&&
& v_{\frac{n}{2}-2,\frac{n}{2}+3}\ar@{-}[rd]\ar@{<-}@/^/[lllldddd]
&v_{\frac{n}{2}-\frac{3}{2},\frac{n}{2}+\frac{7}{2}}\ar@{--}[ld]\ar@{--}[rd]
&\cdots \ar@{-}[ld]
 \\
v_{1,n} \ar@{--}[rd]
&v_{\frac{3}{2},\half} \ar@{-}[ld]
&&&
&
 &
 &\cdots 
 \\
 \ar@{..}[d] &\ar@{..}[d]& &&  &\ar@{..}[d] \ar@{..}[u] &\ar@{..}[d]\ar@{..}[u] &\ar@{..}[d] \ar@{..}[u] 
 \\
\ar@{-}[rd]
&\ar@{--}[ld]
&&&
 &v_{1,n} \ar@{-}[rd]
&v_{\frac{3}{2},\half}\ar@{--}[ld]\ar@{--}[rd]
&\ar@{-}[ld]
 &&
 \\
v_{\frac{n}{2}-2,\frac{n}{2}+3} \ar@{--}[rd]
&v_{\frac{n}{2}-\frac{3}{2},\frac{n}{2}+\frac{7}{2}}\ar@{-}[ld]
&&&
&v_{\half,\half} \ar@{--}[rd]
&v_{1,1} \ar@{-}[ld]\ar@{-}[rd]
&\cdots \ar@{--}[ld]
 \\
1&1&&\cdots &&1&1&\cdots
}
$$

Proposition \ref{Period} is proved.

\subsection{Proof of Theorem \ref{MainOne}}

Given a closed numerical frieze $F(a_i,b_i)$,
let us show that all the solutions of the corresponding
difference equation \eqref{recur} are periodic.
Proposition \ref{Period}, Part~(ii) implies that
all the diagonals $\Db_j$ are $n$-periodic.
Take three consecutive diagonals, say $\Db_1,\Db_2,\Db_3$,
they provide linearly independent periodic solutions
$(v_{i,1},\,v_{i,2},\,v_{i,3})$ to \eqref{recur}.
It follows that every solution is periodic, so that the
monodromy matrix $M$ is the identity.

Conversely, suppose that all the solutions of \eqref{recur} are periodic,
i.e. the monodromy is the identity.
Consider the  frieze $F(a_i,b_i)$.
We have proved that the diagonals $\Delta_i$ with $i\in\Z$ satisfy the recurrence equation
\eqref{rec1}, which is nothing else but \eqref{recur}.
Add formally two rows of zeroes above the first row of 1's:
$$
\begin{matrix}
 \cdots
& 0&0&0&0&0&\cdots
\\[4pt]
 \cdots
& 0&0&0&0&0&\cdots
 \\[4pt]
 \cdots
& 1&1&1&1&1&\cdots
 \\[4pt]
\cdots&b_i&a_i&b_{i+1}&a_{i+1}&b_{i+2}&\cdots
 \\
&\vdots &\vdots &\vdots &\vdots &\vdots &&
\end{matrix}
$$
One checks immediately that this changes nothing in the recurrence relation.
It follows that the diagonals  $\Delta_i$ with $i\in\Z$ are periodic.

The diagonals $\Delta_i$ with $i\in\Z+\half$ satisfy the recurrence equation \eqref{rec2}.
This is precisely the difference equation dual to \eqref{recur}.
The corresponding monodromy is the conjugation of the monodromy of \eqref{recur},
so that it is again equal to the identity.
It follows that the half-integer diagonals are also periodic.

In particular, two consecutive rows of 0's followed by a row of 1's will appear again. These rows are necessarily  preceded by a row 1's in order to satisfy \eqref{rec1} and \eqref{rec2}.
The 2-frieze pattern is closed.


\subsection{Difference equations, and polygons in space and in the projective plane}\label{relations}
In this section, we prove Proposition \ref{OSTProp}.

The relation between the difference equations \eqref{recur} and polygons was already established in 
Section \ref{Analvers} (combined with Section \ref{GeomSec}).
A difference equation has a 3-dimensional space of solutions, 
and these solutions form a sequence of vectors satisfying \eqref{recur}.
If the monodromy is trivial then the sequence of vectors is $n$-periodic.
The sequence of the $V_i$ (as vector in 3-dimensional space) then gives 
a closed polygon (otherwise, it is  twisted). 
It follows from \eqref{recur} that the determinant of every three consecutive vectors is the same. 
One can scale the vectors to render this determinant unit, 
and then the $\SL_3$-equivalence class of the polygon is uniquely determined. 

Conversely, a polygon satisfying the unit determinant condition gives rise 
to a difference equation (\ref{recur}): 
each next vector is a linear combination of the previous three, and one coefficient is equal to 1, 
due to the determinant condition. Two $\SL_3$-equivalent polygons yield the same difference equation, 
and a closed polygon yields an equation with trivial monodromy.

To prove that the projection $\tilde \cC_n\to \cC_n$ is bijective for $n$ not a multiple of 3, 
let  us construct the inverse map. Given an $n$-gon $(v_i)$ in the projective plane, 
let $\tilde V_i$ be some lift of the point $v_i$ to 3-space. We wish to rescale, 
$V_i=t_i \tilde V_i$, so that the unit determinant relation holds: $\det (V_{i-1}, V_i, V_{i+1})=1$ for all $i$. 
This is equivalent to the system of equations
$$
t_{i-1} t_i t_{i+1}=1/\det (\tilde V_{i-1}, \tilde V_i, \tilde V_{i+1})
$$
(the denominators do not vanish because every triple of consecutive vertices of a polygon is not collinear).
This system has a unique solution if $n$ is not  a multiple of 3. 
Furthermore, projectively equivalent polygons in the projective plane yield 
$\SL_3$-equivalent polygons in 3-space. This completes the proof of Proposition \ref{OSTProp}.

\begin{rem} 
\label{sides}
Two points in $\RP^2$ determine not one, but two segments, and a polygon in~$\RP^2$, 
defined as a cyclic collection of vertices, does not automatically have sides, 
i.e., segments connecting consecutive vertices.
For example, three points in general position determine four different triangles,
thought of as triples of segments sharing end-points. 
In contrast, for a polygon $(V_i)$ in $\R^3$, the sides are the segments $V_i V_{i+1}$. 
Thus, using the above described lifting,  one can make a canonical choice of sides of an 
$n$-gon in $\RP^2$, provided that $n$ is not a multiple of 3. 
Changing the orientation of a polygon does not affect the choice of the segments. 
The choice of the segments does not depend on the orientation of $\R^3$ either.
\end{rem}

\subsection{Convex polygons in space and in the projective plane}\label{ConvSubsect}

The proof of Proposition~\ref{PosConvPro} is immediate now.
According to formula~\eqref{MalDetEq}, all the entries of a real $2$-frieze pattern 
are positive if and only if the respective polygon in $\R^3$ is convex, as claimed. 

We shall now discuss the relation between convexity in space and in the projective plane. An $n$-gon $(v_i)$ in $\RP^2$ is called convex if there exists an affine chart in which the closed polygonal line $v_1 v_2 \dots v_n$ is convex. The space of convex $n$-gons in $\RP^2$ is denoted by 
$\cC_n^0$.

\begin{lem} \label{conv}
If $(V_i)$ is a convex $n$-gon in space then its projection to $\RP^2$ is a convex $n$-gon. Conversely, if $n$ is not a multiple of $3$ and $(v_i)$ is a convex $n$-gon in $\RP^2$ then its lift to $\R^3$ is a convex $n$-gon.
\end{lem}

\begin{proof}
Let $(V_i)$ be a convex $n$-gon in space. Let $\pi$ be the oriented plane spanned by $V_1$ and $V_2$. Let $V_1^{\e}$ and $V_2^{\e}$ be points on the negative side of $\pi$ that are $\e$-close to $V_1$ and $V_2$. Let $\pi_{\e}$ be the plane spanned by $V_1^{\e}$ and $V_2^{\e}$. If $\e$ is a sufficiently small positive number, all the points $V_i$ lie on one side of $\pi_{\e}$. Without loss of generality, we may assume that $\pi_{\e}$ is the horizontal plane and all points $V_i$ are in the upper half-space. Consider the radial projection of the polygon $(V_i)$ on the horizontal plane at height one. This plane provides an affine chart of  $\RP^2$. The resulting polygon, $(v_i)$, has the property that, for every $i$ and every $j\ne i-1,i$, the vertex $v_j$ lies on the positive side of the line $v_{i-1} v_i$. Hence this projection is a convex polygon in this plane.

Conversely, let $(v_i)$ be a convex polygon in the projective plane. As before, we assume that the vertices are located in the horizontal plane in $\R^3$ at height one. Convexity implies that the polygon lies on one side of the line through each side, that is, with the proper orientation, 
that $\det (v_{i-1},v_i,v_j)>0$ for all $i$ and $j\ne i-1,i$. One needs to rescale the vectors, $V_i=t_i v_i$, to achieve the unit determinant condition on triples of consecutive vectors. Since the determinants are already positive, it follows that $t_i>0$ for all $i$. Therefore $\det (V_{i-1},V_i,V_j)>0$ for all $i$ and $j\ne i-1,i$, and $(V_i)$ is a convex polygon.
\end{proof}

\subsection{The space $\cC_{3m}$ and the Fock-Goncharov variety}\label{FGSec}

Consider the special case
of the space~$\cC_n$ with $n=3m$.
As already mentioned, in this case, the space $\cC_n$ is not isomorphic to 
the space of difference equations \eqref{recur} and therefore it is not isomorphic to 
the space of closed 2-frieze patterns.
We discuss this special case for the sake of completeness and because it provides an interesting link to another cluster variety.

In \cite{FG, FG1} (see also \cite{OT}, Section 6.5) Fock and Goncharov
introduced and thoroughly studied the space $\cP_n$
consisting of pairs of convex $n$-gons $(P,P')$ in $\RP^2$,
modulo projective equivalence,
such that $P'$ is inscribed into~$P$.

The following statement relates the space $\cP_{2m}$
and the space $\cC^0_{3m}$ of convex
$3m$-gons.

\begin{prop}
\label{FockSpaceProp}
The space $\cC^0_{3m}$ is isomorphic to the space $\cP_{2m}$.
\end{prop}

\begin{proof}
The proof consists of a construction,
see Figure \ref{SecondFig}.
Consider a convex $3m$-gon.

\begin{enumerate}
\item
Choose a vertex $v_i$ and draw the short diagonal $(v_{i-1},\,v_{i+1})$.

\item
Extend the sides $(v_i,\,v_{i+1})$ and $(v_{i+2},\,v_{i+3})$
to their intersection point.

\item
Repeat the procedure starting from $v_{i+3}$.
\end{enumerate}

\noindent
One obtains a pair of $2m$-gons inscribed one into the other.
The procedure is obviously bijective and commutes with the $\SL(3,\R)$-action.
\end{proof}

\begin{figure}[hbtp]
\includegraphics[width=10cm]{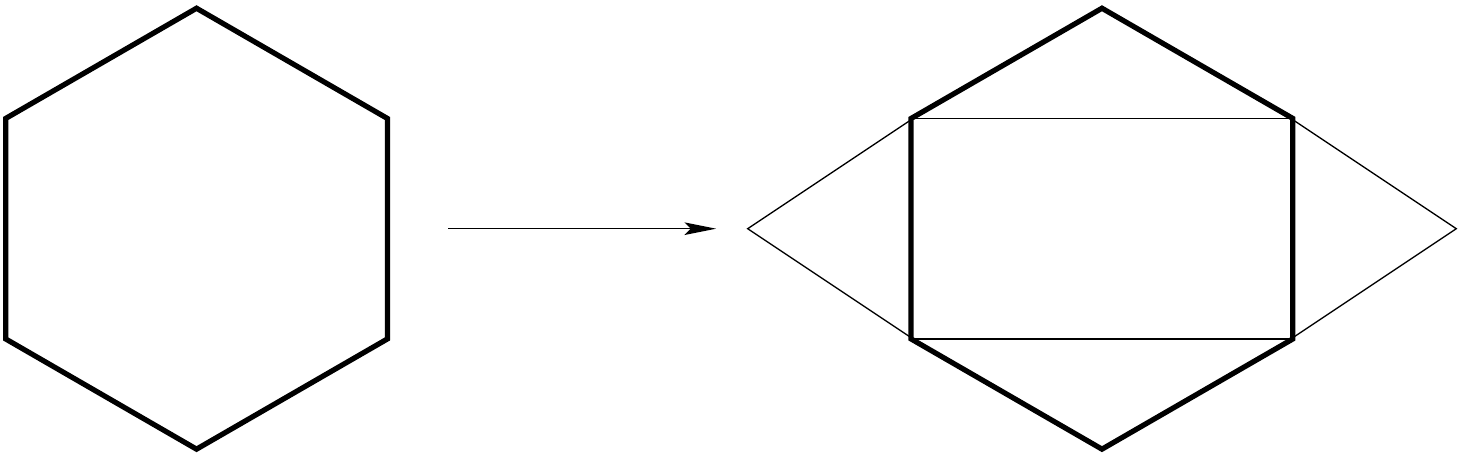}
\caption{From a hexagon to a pair of inscribed quadrilaterals}
\label{SecondFig}
\end{figure}

\begin{rem}
Choosing the vertex $v_{i+1}$ or $v_{i+2}$ in the above construction,
one changes the identification between the spaces $\cC^0_{3m}$ and $\cP_{2m}$.
This, in particular, defines a map $\tau$ from $\cP_{2m}$ to $\cP_{2m}$, such that $\tau^3=\Id$.
\end{rem}

\section{Closed 2-friezes as cluster varieties}\label{ClustSect}

We now give a description of the space
of all closed 2-frieze patterns.
This is an 8-codimen\-sional subvariety
of the space $\C^{2n}$ (or $\R^{2n}$) identified with the space
of $2n$-periodic patterns.
We will characterize this variety using the technique of cluster
manifolds.

\subsection{Cluster algebras }

Let us recall the construction of Fomin-Zelevinsky's cluster algebras~\cite{FZ1}.
A cluster algebra $\A$ is a commutative associative algebra.
This is a subalgebra of a field of rational fractions
in $N$ variables, where $N$ is called the rank of $\A$.
The algebra $\A$ is presented by generators and relations. 
The generators are collected in packages called \textit{clusters}, the 
relations between generators are obtained by applying a series 
of specific elementary relations called the \textit{exchange relations}. 
The exchange relations are encoded via a matrix, or an oriented graph
with no loops and no $2$-cycles.

The explicit construction of the (complex or real) cluster algebra  $\A(\Qc)$
associated to a finite oriented graph $\Qc$ is as follows.
Let $N$ be the number of vertices of $\Qc$,
the set of vertices is then identified with the set $\{1, \ldots, N\}$.
The algebra $\A(\Qc)$ is a subalgebra of the field of fractions $\C(x_1,\ldots, x_N)$ in $N$  
variables $x_1,\ldots, x_N$ (or over $\R$, in the real case). 
The generators and relations of $\A(\Qc)$ are given using a recursive procedure called
\textit{seed mutations} that we describe below.

A \textit{seed} is a couple 
$$
\Sigma=\left(\{t_1, \ldots, t_N\} , \;\Rc\right),
$$
where $\Rc$ is an arbitrary finite oriented graph with $N$ vertices
and where $t_1, \ldots, t_N$ are free generators of $\C(x_1,\ldots, x_N)$. 
The \textit{mutation at  vertex}  $k$ of the seed 
$\Sigma$ is a new seed $\mu_k(\Sigma )$ defined by
\begin{enumerate}
\item[\textbullet]
$\mu_k(\{t_1, \ldots, t_N\})=\{t_1, \ldots, t_{k-1},t'_k, t_{k+1},\ldots, t_N\}$ where
$$\displaystyle
t'_k=\dfrac{1}{t_k}\left(\prod\limits_{\substack{\text{arrows in }\Rc\\ i\rightarrow k }}\; t_i
 \quad+\quad 
\prod\limits_{\substack{\text{arrows in }\Rc\\ i\leftarrow k }}\;t_i\right)
$$
\item[\textbullet] 
$\mu_k(\Rc)$ is the graph obtained from $\Rc$ by applying the following transformations
\begin{enumerate}
\item for each possible path $i\rightarrow k \rightarrow j$ in $\Rc$, add an arrow $i\rightarrow j$,
\item reverse all the arrows leaving or arriving at $k$,
\item remove all the possible 2-cycles, 
\end{enumerate}
\end{enumerate}
(see Example \ref{exmut}  below for a seed mutation).

Starting from the initial seed $\Sigma_0=(\{x_1,\ldots, x_N\}, \Qc)$, one produces  $N$ new seeds 
$\mu_k(\Sigma_0)$, $k=1,\ldots, N$. 
Then one applies all the possible mutations to all of the created new seeds, and so on.
The set of rational functions appearing in any of the seeds produced during the mutation process
is called a \textit{cluster}. 
The functions in a cluster are called \textit{cluster variables}.
The cluster algebra $\A(\Qc)$ is the subalgebra of  $\C(x_1,\ldots, x_N)$ generated by all the cluster variables.

\begin{ex}\label{exmut}
In the case $n=4$, consider the seed $\Sigma=(\{t_1,t_2,t_3,t_4\},\Rc)$, where

\xymatrix{&&&&
\Rc=
&1\ar@{->}[r]
&2\ar@{->}[d]
\\
&&&&
&3\ar@{<-}[r]\ar@{->}[u]
&4
}
The mutation at vertex 1 gives 
$$
\mu_1(\{t_1,t_2,t_3,t_4\})=\Big\{\frac{t_2+t_3}{t_1},t_2,t_3,t_4\Big\}
$$ and 

\xymatrix{&&&&
\mu_1(\Rc)=
&1\ar@{<-}[r]
&2\ar@{->}[d]\ar@{<-}[ld]
\\
&&&&
&3\ar@{<-}[r]\ar@{<-}[u]
&4
}

\medskip
In this example, one can show that the mutation process is finite. 
This means that applying all the possible mutations to all the seeds leads to a finite number of seeds and therefore to a finite number (24) of cluster variables.
One can also show that among the graphs obtained through iterated mutations is the Dynkin graph of type $D_4$. 
The cluster algebra $A(\Rc)$ in this example is referred to as the cluster algebra of type $D_4$.
\end{ex}

\subsection{The algebra of regular functions on $\F_n$}

In the case of the oriented graph \eqref{Graph}, the cluster algebra $\A(\Qc)$
has an infinite number of generators (for $n\geq9$).
In this section, we consider the algebra of regular functions
on $\F_n$ and show that this is a subalgebra of $\A(\Qc)$.
>From now on, $\Qc$ always stands for the oriented graph \eqref{Graph}.

The space of closed 2-friezes $\F_n$ is an algebraic manifold, $\F_n\subset\C^{2n}$
(or $\R^{2n}$ in the real case), defined by the trivial monodromy
condition $M=\Id$, that can be written as 8 polynomial identities.
The algebra of regular functions on $\F_n$ is then defined as
$$
\A_n=
\C[A_1,\ldots,A_n,B_1,\ldots,B_n]/\mathcal{I},
$$
where $\mathcal{I}$ is the ideal generated by $(M-\Id)$.
Let us describe the algebra $\A_n$ in another way.

We define the following system of coordinates on the space $\F_n$.
Consider $2n-8$ independent variables $(x_1,\ldots, x_{n-4},y_1,\ldots, y_{n-4})$
and place them into two consecutive columns on the frieze:
\begin{equation}
\label{ClustF}
 \begin{matrix}
& 1&1&1&1&1&\cdots
 \\[4pt]
&x_1&y_1&\cdots&&&
 \\[4pt]
&y_2&x_2&\cdots&&&
 \\[4pt]
 &x_3&y_3&\cdots&&&
 \\[4pt]
 &y_4&x_4&\cdots&&&
 \\
&\vdots &\vdots & & & &&\\
& 1&1&1&1&1&\cdots
\end{matrix}
\end{equation}
Applying the recurrence relations, complete the 2-frieze pattern by
rational functions in $x_i,y_j$.
Since the 2-frieze pattern (\ref{ClustF}) is closed,
Proposition \ref{Period} implies that 
the closed 2-frieze pattern (\ref{ClustF}) contains $n(n-4)$ distinct entries
modulo periodicity.

\begin{ex}
\label{ClustFive}
Case $n=5$
$$
\begin{array}{cccccccccccccccc}
\cdots &1 & 1 & 1 & 1& 1& 1& 1& \cdots\\[6pt]
\cdots & x & y&   \frac{y+1}{x}&  \frac{x+y+1}{xy}& \frac{x+1}{y}&x & y&\cdots \\[6pt]
\cdots &1& 1& 1& 1& 1& 1& 1& \cdots \\
\end{array}
$$
In this case, $\A_5\simeq \A(1\rightarrow2)$.
\end{ex}

\begin{ex}
\label{ClustSix}
Case $n=6$
\begin{equation*}
\begin{array}{cccccccccccccccc}
\cdots &1 & 1 & 1 & 1& 1& 1&  1& 1& \cdots\\[6pt]
\cdots & x_1 & y_1&   \frac{y_1+x_2}{x_1}& \frac{(y_1+x_2)(y_2+x_1)}{x_1y_1y_2}
& \frac{(x_1+y_2)(x_2+y_1)}{x_2y_1y_2} & \frac{x_1+y_2}{x_2}&y_2 & x_2&\cdots \\[6pt]
\cdots & y_2 & x_2&   \frac{x_2+y_1}{y_2}& \frac{(x_2+y_1)(x_1+y_2)}{y_2x_2x_1}
& \frac{(y_2+x_1)(y_1+x_2)}{y_1x_1x_2} & \frac{y_2+x_1}{y_1}&x_1 & y_1&\cdots \\[6pt]
\cdots &1& 1& 1& 1& 1& 1& 1& 1& \cdots \\
\end{array}
\end{equation*}
In this case, the algebra $\A_6$ is isomorphic to a proper subalgebra of $ \A(D_4)$.
\end{ex}

\begin{prop}
\label{ComAlgProp}
(i)
The algebra $\A_n$ is isomorphic to the subalgebra of
the algebra of rational functions
$\C(x_1,\ldots,x_{n-4},y_1,\ldots,y_{n-4})$
generated by the entries of the 2-frieze \eqref{ClustF}.

(ii)
The algebra $\A_n$ is a subalgebra of the cluster algebra
$\A(\Qc)$, where $\Qc$ is the graph \eqref{Graph}.
\end{prop}

\begin{proof}
(i)
The entries of \eqref{ClustF} are polynomials in $2n$ consecutive
entries of the first row (see Proposition \ref{TabProp}).
The isomorphism is then obtained by sending $B_1,A_1,\ldots,B_n,A_n$
to the entries of the first line.

(ii)
Consider $\Sigma_0=(\{x_1,\ldots, x_{n-4},y_1,\ldots, y_{n-4}\}, \Qc)$ as an initial seed.
The variable $x_i$ is associated to the vertex $i$ of $\Qc$ 
and the variable $y_i$ to the vertex $n-4+i$.
We need to prove that all the entries of \eqref{ClustF} are cluster variables.

The graph $\Qc$ is bipartite. 
One can associate a sign $\e(i)=\pm$ to each vertex 
of the graph so that any two connected vertices in
$\Qc$ have different signs. Let us assume that $\e(1)=+$ 
(this determines automatically all the signs of the vertices).

Following Fomin-Zelevinsky \cite{FZ4},
consider the iterated mutations
$$
\mu_+=\prod_{i: \e(i)=+}\;\mu_i, \qquad 
\mu_-=\prod_{i:\e(i)=-}\;\mu_{i}.
$$
Note that $\mu_i$ with $\e(i)$ fixed commute with each other.

It is important to notice that
the result of the mutation of the graph \eqref{Graph} by $\mu_+$ and $\mu_-$
is the same graph with reversed orientation:
$$
\mu_+(\Qc)=\Qc^{\hbox{op}},
\qquad
\mu_-(\Qc^{\hbox{op}})=\Qc.
$$
This is a straightforward verification.

Consider the seeds of $\A(\Qc)$
obtained from $\Sigma_0$ by applying successively $\mu_+$ or $\mu_{-}$:
\begin{equation}
\label{BeltB}
\Sigma_0,\quad
\mu_+(\Sigma_0),\quad \mu_-\mu_+(\Sigma_0),\quad \ldots, \quad
\mu_{\pm}\mu_{\mp}\cdots \mu_-\mu_+(\Sigma_0),\quad \ldots
\end{equation}
This set is called the bipartite belt of $\A(\Qc)$, see \cite{FZ4}.
The cluster variables in each of the above seeds correspond
precisely to two consecutive columns in the 2-frieze pattern \eqref{ClustF}.

Proposition \ref{ComAlgProp} is proved.
\end{proof}

\begin{rem}
Periodicity of the sequence \eqref{BeltB} follows from
Proposition \ref{Period}.
This periodicity of closed 2-frieze patterns, expressed in cluster variables,
is a particular case of the general periodicity theorem in cluster algebra,
see \cite{Volk, Kel} and references therein.
Our proof of this result is based on simple properties of solutions
of the difference equation~\eqref{recur} and is given for the sake of completeness.
\end{rem}

\subsection{Zig-zag coordinates}\label{FunClo}
In this section, we prove
Theorem \ref{MainSecond} and Proposition \ref{Boundme}.
To this end, we introduce a number of coordinate systems on the space of 2-friezes.

We define another system of coordinates on $\F_n$.
Draw an arbitrary \textit{double zig-zag} in the 2-frieze \eqref{ClustF}
and denote by
$(\widetilde{x}_1,\ldots, \widetilde{x}_{n-4},\widetilde{y}_1,\ldots, \widetilde{y}_{n-4})$
the entries lying on this double zig-zag:
\begin{equation}
\label{DobZag}
 \begin{matrix}
\cdots & 1&1&1&1&1&\cdots
 \\[4pt]
&&\widetilde{x}_1&\widetilde{y}_1&&&
 \\[4pt]
&&&\widetilde{x}_2&\widetilde{y}_2&&
 \\[4pt]
& &\widetilde{x}_3&\widetilde{y}_3&&&
 \\
&&\vdots &\vdots & & & &\\
\cdots& 1&1&1&1&1&\cdots
\end{matrix}
\end{equation}
in such a way that
$\widetilde{x}_i$ stay at the entries with integer indices
and $\widetilde{y}_i$ stay at the entries with half-integer indices.

More precisely, a double zig-zag of coordinates is defined as follows.
The coordinates 
$\widetilde{x}_i$ and~$\widetilde{y}_i$ in the $i$-th row, are followed by
the coordinates $\widetilde{x}_{i+1}$ and $\widetilde{y}_{i+1}$ in one of the
three possible ways:
$$
\xymatrix{
&\widetilde{x}_i\ar@{->}[ld]
&\widetilde{y}_i\ar@{->}[ld]
&&
\widetilde{x}_i\ar@{->}[rd]
&\widetilde{y}_i\ar@{->}[ld]
&&
\widetilde{x}_i\ar@{->}[rd]
&\widetilde{y}_i\ar@{->}[rd]
\\
\widetilde{x}_{i+1}
&
\widetilde{y}_{i+1}
&&&
\widetilde{y}_{i+1}
&
\widetilde{x}_{i+1}
&&&
\widetilde{x}_{i+1}
&
\widetilde{y}_{i+1}
}
$$

Denote by $\cZ$ the set of all double zig-zags.
For an arbitrary double zig-zag $\zeta\in\cZ$, the 
corresponding functions  
$(\widetilde{x}_1,\ldots, \widetilde{x}_{n-4},\widetilde{y}_1,\ldots, \widetilde{y}_{n-4})_\zeta$
are rational expressions in $(x_i,y_i)$.

\begin{prop}
\label{MutProp}
(i)
For every double zig-zag $\zeta\in\cZ$,
the coordinates $(\widetilde{x}_i, \widetilde{y}_i)_\zeta$
form a cluster in the algebra $\A(\Qc)$, where $\Qc$ is the graph \eqref{Graph}.

(ii)
A cluster in $\A(\Qc)$ coincides with the coordinate system
$(\widetilde{x}_i, \widetilde{y}_i)_\zeta$
for some $\zeta\in\cZ$, if and only if it is
obtained from the initial cluster $(x_i,y_i)$ by 
mutations at vertices that do not  belong to two
3-cycles.
\end{prop}

\begin{proof}
(i)
For every double zig-zag $\zeta$,
we define a seed $\Sigma_\zeta=\left((x_i,y_i)_\zeta,\Qc_\zeta\right)$
in the algebra $\A(\Qc)$, where $\Qc_\zeta$ is
the oriented graph associated to $\zeta$ defined
as follows.

The fragments of zig-zags:
$$
\begin{array}{rrr}
&\widetilde{x}_i&\widetilde{y}_i\\[6pt]
\widetilde{x}_{i+1}&\widetilde{y}_{i+1}&
\end{array}
\qquad
\qquad
\begin{array}{ll}
\widetilde{x}_i&\widetilde{y}_i\\[6pt]
\widetilde{y}_{i+1}&\widetilde{x}_{i+1}
\end{array}
\qquad
\qquad
\begin{array}{lll}
\widetilde{x}_i&\widetilde{y}_i&\\[6pt]
&\widetilde{x}_{i+1}&\widetilde{y}_{i+1}
\end{array}
$$
correspond, respectively, to the following subgraphs:
$$
\!\!\!\!\!\!\!\!\!\!\!\!\!\!\!\!\!\!
\xymatrix{
&\widetilde{x}_i\ar@{->}[r]\ar@{<-}[rd]
&\widetilde{x}_{i+1}\ar@{->}[d]
\\
&\widetilde{y}_i\ar@{->}[r]\ar@{<-}[u]
&\widetilde{y}_{i+1}
}
\quad
\xymatrix{
&\widetilde{x}_i\ar@{<-}[r]
&\widetilde{x}_{i+1}\ar@{<-}[d]
\\
&\widetilde{y}_i\ar@{->}[r]\ar@{<-}[u]
&\widetilde{y}_{i+1}
}
\xymatrix{
&\widetilde{x}_i\ar@{<-}[r]
&\widetilde{x}_{i+1}\ar@{->}[d]\ar@{<-}[ld]
\\
&\widetilde{y}_i\ar@{<-}[r]\ar@{<-}[u]
&\widetilde{y}_{i+1}
}
$$
for $i$ even and with reversed orientation for $i$ odd.
Similarly, the fragments
$$
\begin{array}{rrr}
&\widetilde{y}_i&\widetilde{x}_i\\[6pt]
\widetilde{y}_{i+1}&\widetilde{x}_{i+1}&
\end{array}
\qquad
\qquad
\begin{array}{ll}
\widetilde{y}_i&\widetilde{x}_i\\[6pt]
\widetilde{x}_{i+1}&\widetilde{y}_{i+1}
\end{array}
\qquad
\qquad
\begin{array}{lll}
\widetilde{y}_i&\widetilde{x}_i&\\[6pt]
&\widetilde{y}_{i+1}&\widetilde{x}_{i+1}
\end{array}
$$
correspond to
$$
\!\!\!\!\!\!\!\!\!\!\!\!\!\!\!\!\!\!
\xymatrix{
&\widetilde{x}_i\ar@{->}[r]
&\widetilde{x}_{i+1}\ar@{<-}[d]\ar@{->}[ld]
\\
&\widetilde{y}_i\ar@{->}[r]\ar@{->}[u]
&\widetilde{y}_{i+1}
}
\quad
\xymatrix{
&\widetilde{x}_i\ar@{->}[r]
&\widetilde{x}_{i+1}\ar@{->}[d]
\\
&\widetilde{y}_i\ar@{<-}[r]\ar@{->}[u]
&\widetilde{y}_{i+1}
}
\quad
\xymatrix{
&\widetilde{x}_i\ar@{<-}[r]\ar@{->}[rd]
&\widetilde{x}_{i+1}\ar@{<-}[d]
\\
&\widetilde{y}_i\ar@{<-}[r]\ar@{->}[u]
&\widetilde{y}_{i+1}
}
$$
for $i$ even and with reversed orientation for $i$ odd.
Applying this recurrent procedure,
one defines an oriented graph $\Qc_\zeta$.

For every double zig-zag $\zeta$, there is a series of zig-zags
$\zeta_1,\zeta_2,\ldots,\zeta_k$ such that $\zeta_i$ and~$\zeta_{i+1}$
differ in only one place, say $\widetilde{x}_{\ell_i},\widetilde{y}_{\ell_i}$,
and such that $\zeta_k$ is the double column \eqref{ClustF}.
It is easy to check that
every ``elementary move'' $\zeta_i\to\zeta_{i+1}$ is obtained by
a mutation of coordinates $(\widetilde{x}_i, \widetilde{y}_i)_{\zeta_i}$,
while the corresponding graph $\Qc_{\zeta_{i+1}}$ is a mutation of $\Qc_{\zeta_i}$.

(ii)
Every graph $\Qc_\zeta$ that we construct in the seeds corresponding to
zig-zag coordinates is of the form
\begin{equation}
\label{Qz}
\xymatrix{
&1\ar@{->}[r]
&2\ar@{<-}[r]\ar@{->}[d]
&3\ar@{->}[r]\ar@{<-}[ld]\ar@{<-}[rd]
&\cdots
&\cdots \ar@{<-}[r]
&n-5\ar@{->}[r]
&n-4\ar@{->}[d]\\
&
n-3\ar@{<-}[r]\ar@{->}[u]
&n-2\ar@{<-}[r]
&n-1\ar@{<-}[u]\ar@{->}[r]
&\cdots
&\cdots \ar@{->}[r]
&2n-9\ar@{->}[u]\ar@{<-}[r]
&2n-8\\
}
\end{equation}
That is, $\Qc_\zeta$ is the initial graph \eqref{Graph} with some
diagonals added (such that the triangles and empty squares
are cyclically oriented).
Conversely, from every such graph, one immediately constructs a double zig-zag.

A graph of the form \eqref{Qz} can be obtained from the initial graph $\Qc$ by
a series of mutations at the vertices that do not  belong to triangles.
Conversely, a mutation at a vertex on a triangle changes the nature of the graph
(it removes sides of squares).
\end{proof}

\begin{ex}
Consider a double-diagonal, it can be ``redressed'' to a double-column by a series of elementary moves,
for instance in the case $n=7$,
$$
\begin{array}{rrrr}
&&\widetilde{y}_1&\widetilde{x}_1\\[4pt]
&\widetilde{y}_2&\widetilde{x}_2&\\[4pt]
\widetilde{y}_3&\widetilde{x}_3&&
\end{array}
\quad
\rightarrow
\quad
\begin{array}{rll}
&{\widetilde{x}_1}^\prime&\widetilde{y}_1\\[4pt]
&\widetilde{y}_2&\widetilde{x}_2\\[4pt]
\widetilde{y}_3&\widetilde{x}_3&
\end{array}
\quad
\rightarrow
\quad
\begin{array}{ll}
{\widetilde{x}_1}^\prime&\widetilde{y}_1\\[4pt]
\widetilde{y}_2&\widetilde{x}_2\\[4pt]
\widetilde{x}_3&\widetilde{y}_3^\prime
\end{array}
$$
The corresponding graphs are:
$$
\!\!\!\!\!\!\!\!\!\!\!\!\!\!\!\!\!\!
\xymatrix{
&1\ar@{<-}[r]
&2\ar@{->}[d]\ar@{<-}[r]\ar@{<-}[ld]
&3\ar@{->}[d]\ar@{<-}[ld]&&
\\
&4\ar@{<-}[r]\ar@{<-}[u]
&5\ar@{<-}[r]
&6&\ar@<20pt>@{->}[r]^{\mu_1}&
}
\quad
\xymatrix{
1\ar@{->}[r]
&2\ar@{->}[d]\ar@{<-}[r]
&3\ar@{->}[d]\ar@{<-}[ld]&&
\\
4\ar@{<-}[r]\ar@{->}[u]
&5\ar@{<-}[r]
&6&\ar@<20pt>@{->}[r]^{\mu_6}&
}
\quad
\xymatrix{
1\ar@{->}[r]
&2\ar@{->}[d]\ar@{<-}[r]
&3\ar@{<-}[d]
\\
4\ar@{<-}[r]\ar@{->}[u]
&5\ar@{->}[r]
&6
}
$$
\end{ex}

We are ready to prove Theorem \ref{MainSecond}.

Part (i). It follows from the Laurent phenomenon for cluster
algebras \cite{FZ2} that all the entries of the frieze (\ref{DobZag})
are Laurent polynomials in any zig-zag coordinates $(\widetilde{x}_i, \widetilde{y}_i)_\zeta$.
Therefore, for every double zig-zag $\zeta$, we obtain a 
well-defined map from the complex torus~$(\C^*)^{2n-8}$
to the open dense subset of $\F_n$ consisting of 2-friezes with non-vanishing entries
on $\zeta$.
Hence the result.

Part (ii).
Assume that the coordinates $(\widetilde{x}_i, \widetilde{y}_i)_\zeta$ are positive real numbers.
It then follows from the 2-frieze rule that all the entries of the frieze are positive.
Therefore, every system of coordinates $(\widetilde{x}_i, \widetilde{y}_i)_\zeta$
identifies the subspace $\F^0_n$ with $\R_{>0}^{2n-8}$.

Theorem \ref{MainSecond} is proved.

Let us also prove Proposition \ref{Boundme}.
Consider an arbitrary double zig-zag $\zeta$.
The set of coordinates 
$(\widetilde{x}_1,\ldots, \widetilde{x}_{n-4},\widetilde{y}_1,\ldots, \widetilde{y}_{n-4})_\zeta$
forms a cluster. 
Proposition \ref{Boundme} then follows from the Laurent phenomenon.

\subsection{The cluster manifold of closed 2-friezes}\label{NoLabel}
The two systems of coordinates 
$(\widetilde{x}_i, \widetilde{y}_i)_\zeta$ and $(\widetilde{x}_i, \widetilde{y}_i)_{\zeta'}$,
where $\zeta$ and $\zeta'$ are two double zig-zags,
can be reached from one another
by a series of mutations.

Consider all the coordinate systems corresponding to different double zig-zags.
We call the \textit{cluster manifold of 2-friezes} the smooth analytic (complex) manifold
obtained by gluing together the complex tori $(\C^*)^{2n-8}$
via the consecutive mutations.

The cluster manifold of 2-friezes is not the
entire algebraic variety $\F_n$.
Indeed, the smooth cluster manifold of 2-friezes consists of the 2-friezes that have at least
one double zig-zag with non-zero entries.
However, the full space $\F_n$ also contains singular points.

To give an example, consider $n=km$ with $k,m\geq3$
and take an $n$-gon obtained as an $m$-gon traversed $k$ times.
The corresponding closed 2-frieze pattern (of width $n-4$) contains
double rows of zeroes
(this readily follows from formula \eqref{MalDetEq}).
This 2-frieze pattern does not belong to the smooth cluster manifold of 2-friezes.

\begin{rem}
In the cases $n=6,7,8$, Scott \cite{Sco} proved that the cluster algebra built out of the graph \eqref{Graph} 
is isomorphic to the coordinate ring of the Grassmannian $Gr(3,n)$.
In these algebras, Pl\"ucker coordinates form a proper subset of the set of cluster variables
(14  Pl\"ucker coordinates among 16 cluster variables for $n=6$, 
28 among 42 for $n=7$ and 48 among 128 for $n=8$).
It can be checked that the $n(n-4)$ cluster variables arising in the 2-frieze of width $n-4$
are also Pl\"ucker coordinates.
However, we do not know if there is a nice way to characterize
them using Scott's approach
(in terms of Postnikov arrangement or root correspondence). 
\end{rem}

\subsection{The symplectic structure}\label{Symplect}

The complete integrability of the pentagram map \cite{OST} was deduced from the existence of an
invariant Poisson structure on the space of twisted $n$-gons.
In general, completely integrable dynamics is
usually associated with invariant symplectic or Poisson structure - thus our interest in
this question.

Every cluster manifold has a canonical (pre)symplectic form,
i.e., a closed differential 2-form, see \cite{GSV}.
Let us recall here the general definition.
For an arbitrary seed $\Sigma=\left(\{t_1, \ldots, t_N\} , \;\Rc\right)$
on a cluster manifold, the 2-form is as follows:
\begin{equation}
\label{GenSym}
\omega=
\sum\limits_{\substack{\text{arrows in }\Rc\\ i\rightarrow j }}
\frac{dt_i}{t_i}\wedge\frac{dt_j}{t_j}.
\end{equation}
It is then easy to check that the 2-form $\omega$ is well-defined,
that is, does not change under mutations.
The 2-form $\omega$ is obviously closed (since it is constant in the
coordinates $\log{}t_i$).
However, this form is not always symplectic and may be degenerate.

One of the consequence of the defined cluster structure
 is the existence of such a form on the cluster manifold of 2-friezes.
 It turns out that this form is \textit{non-degenerate},
 for $n\not=3m$.
 
 \begin{prop}
 \label{SymplProp}
 (i)
 The differential $2$-form \eqref{GenSym} on $\F_n$ is
 symplectic if and only if $n\not=3m$.
 
 (ii) If $n=3m$, then the form \eqref{GenSym} is a presymplectic form of corank $2$.
 \end{prop}
 
 \begin{proof}
 It suffices to check the statement for the initial seed
 $\Sigma_0=(\{x_1,\ldots, x_{n-4},y_1,\ldots, y_{n-4}\}, \Qc)$,
 where $\Qc$ is the graph \eqref{Graph}.
 The form \eqref{GenSym} then corresponds to the following skew-symmetric
 $(2n-8)\times(2n-8)$-matrix
 $$
 \omega(n)=
 \left(
 \begin{array}{rrrrrrrr}
 0&1&0&\;\;0&\cdots&0&0&-1\\[4pt]
 -1&0&-1&0&\cdots&0&1&0\\[4pt]
 0&1&0&1&\cdots&-1&0&0\\[4pt]
 &&&&\ddots&&&\\[4pt]
 0&0&1&0&\cdots&0&-1&0\\[4pt]
  0&-1&0&0&\cdots&1&0&1\\[4pt]
 1&0&0&0&\cdots&0&-1&0
 \end{array}
 \right)
 $$
 which is nothing other than the incidence matrix of the graph \eqref{Graph}
 (for technical reasons we inverse the labeling of the second line).
We need to check that this matrix is non-degenerate if and only if $n\not=3m$
 and has corank 2 otherwise.
 
 (i) We proceed by induction on $n$.
 First, one easily checks the statement for small $n$.
 Indeed, for $n=5$, $n=6$ and $n=7$, the matrix $\omega(n)$ is as follows:
$$
\left(
 \begin{array}{rr}
 0&1\\[4pt]
 -1&0
  \end{array}
 \right)
 \quad , 
 \quad
\left(
 \begin{array}{rrrr}
 0&1&0&-1\\[4pt]
 -1&0&1&0\\[4pt]
 0&-1&0&1\\[4pt]
 1&0&-1&0
 \end{array}
 \right)
 \quad\hbox{and}\quad
 \left(
 \begin{array}{rrrrrr}
 0&1&0&0&0&-1\\[4pt]
 -1&0&-1&0&1&0\\[4pt]
 0&1&0&-1&0&0\\[4pt]
 0&0&1&0&-1&0\\[4pt]
  0&-1&0&1&0&1\\[4pt]
 1&0&0&0&-1&0
 \end{array}
 \right)
 $$
respectively.
 
 Next, the matrix $\omega(n)$ is non-degenerate if and only if
 $\omega(n-3)$ is non-degenerate.
 Indeed, denote by $N=2n-8$ the size of the matrix $\om(n)$.
 Add the columns $(N-2)$ and $N$ to the column 2, and add rows $(N-2)$ and $N$ to row 2. 
 One obtains
  $$
 \left(
 \begin{array}{rrr|ccc|rrr}
 0&0&0& &&  &0 &0&-1\\[4pt]
 0&0&0& && -1 &0 &-1&0\\[4pt]
 0&0&0&1&& &-1&0&0\\[4pt]
 \hline
    &&-1&  && & && \\[4pt]
    &&    & &-\omega(n-3)&&&&\\[4pt]
     &1&   &  && &1 && \\[4pt] 
 \hline
 0&0&1& && -1 &0&-1&0\\[4pt]
 0&1&0&&&&1&0&1\\[4pt]
 1&0&0&&&&0&-1&0
 \end{array}
 \right)
 $$
Then, one can subtract column 2 from column $N-2$, add column 1 to column $N-1$ and do similar operations on the rows. 
This leads to a block of zeroes in the right down corner. 
Then one can easily remove the extra $\pm1$'s to finally obtain
 $$
 \left(
 \begin{array}{rcr|ccc|rcr}
 0&0&0& &&  &0 &0&-1\\[4pt]
 0&0&0& &(0)& &0 &-1&0\\[4pt]
 0&0&0&&& &-1&0&0\\[4pt]
 \hline
    &&&  && & && \\[4pt]
    &(0)&    & &-\omega(n-3)&&&(0)&\\[4pt]
     &&   &  && & && \\[4pt] 
 \hline
 0&0&1& && &0&0&0\\[4pt]
 0&1&0&&(0)&&0&0&0\\[4pt]
 1&0&0&&&&0&0&0
 \end{array}
 \right)
 $$

The result follows.
 
 (ii) Let now $n=3m$.
 The  $(2n-10)\times(2n-10)$-minor:
 $$
 \left(\omega(n)_{ij}\right),
 \qquad
 2\leq{}i,j\leq2n-9
 $$
coincides with the matrix $-\omega(n-1)$ which is non-degenerate as already proved in Part (i).
 Therefore, the matrix $\omega(n)$ is, indeed, of corank~2.
 \end{proof}

\begin{rem}
In the case $n=3m$, one can explicitly find a linear combination of the rows
 of~$\omega(n)$ that vanishes:
 $$
 \sum_{0\leq{}i<[m/2]}\left(\ell_{6i+1}+\ell_{N-6i-1}\right)-
 \sum_{1\leq{}i<[m/2]}\left(\ell_{6i-1}+\ell_{N-6i+3}\right),
 $$
 where $N=2n-8$,
 so that $\omega(n)$ is, indeed, degenerate.
 \end{rem}
 
\section{Arithmetic 2-friezes}\label{Trick}

We consider now closed numerical 2-friezes whose
 entries are positive integers, that is, arithmetic 2-friezes.
The problem of classification  of such 2-frieze patterns was formulated in~\cite{Pro}
and interpreted as a generalization of the Catalan numbers.
The problem remains open.

In this section, we describe a stabilization process that is a  step
toward solution of this problem.
It is natural to consider a 2-frieze pattern that can be obtained by
stabilization as ``trivial''.
We thus formulate a problem of classification
of those patterns that cannot be obtained this way. Likewise, it is natural to call a 2-frieze pattern {\it prime} if it is not the connected sum of non-trivial 2-frieze patterns. The classification of prime arithmetic 2-friezes is also a challenging problem.

It was shown in \cite{CoCo} that every (classical Coxeter-Conway) arithmetic frieze pattern
contains~$1$ in the first non-trivial row and
can be obtained by a simple procedure from a pattern of lower width.
This provides a complete classification of Coxeter-Conway.
Our stabilization is quite similar to the classical Coxeter-Conway stabilization.
However, unlike the classical case, classification of 2-frieze patterns does not
reduce to stabilization (cf. for instance Example \ref{NonStab}).

We start this section with the simplest examples.

\subsection{Arithmetic 2-friezes for $n=4,5$}

The case $n=4$ is the first case where the notion
of 2-frieze pattern makes sense.
The unique $8$-periodic pattern is the following one
\begin{equation}
\label{Triv}
\begin{array}{cccccccccc}
\cdots&1&1&1&1&1& 1&1&1&\cdots\\[2pt]
\cdots&1&1&1&1&1&1& 1&1&\cdots
\end{array}
\end{equation}
which is the most elementary 2-frieze pattern.

If $n=5$, the answer is as follows.

\begin{prop}
\label{n=5Prop}
The $2$-frieze pattern
\begin{equation}
\label{OldFriend}
 \begin{array}{cccccccccccc}
\cdots&1&1&1&1&1& 1&1&1&1&1&\cdots\\[4pt]
\cdots&1&1&2&3&2&1&1&2&3&2&\cdots\\[4pt]
\cdots&1&1&1&1&1&1& 1&1&1&1&\cdots
\end{array}
\end{equation}
is the unique arithmetic $2$-frieze pattern of width 1.
\end{prop}

\begin{proof}
According to Proposition \ref{Period}, Parts (ii), (iii),
an integral $2$-frieze pattern of width 1 is of the form
$$
\begin{array}{rrrrrrrrrrrrrrrr}
\cdots&1&1&1&1&1&1&1& \cdots\\
\cdots&b_0&a_0&b_1&a_1&b_2&b_0&a_0&\cdots\\
\cdots&1&1&1&1&1&1&1& \cdots
\end{array}
$$
Let us show that every number $\{b_0,a_0,b_1,a_1,b_2\}$
is less than or equal to 3.

One has from \eqref{ClustFive}:
$$
b_1=\frac{a_0+1}{b_0},
\qquad
b_2=\frac{b_0+1}{a_0}.
$$
Therefore, there exist positive integers $k,\ell$ such that
$b_0+1=k\,a_0$ and $a_0+1=\ell\,b_0$.
Hence
\begin{equation}
\label{a0}
a_0=\frac{\ell+1}{k\,\ell-1}.
\end{equation}
Assume $a_0>3$, then $\ell\,(3k-1)<4$.
Since $k,\ell$ are positive integers, the only possibility is
$k=\ell=1$.
This contradicts \eqref{a0}.

Once one knows that the entries do not exceed 3, the proof is completed by a brief exhaustive search.
\end{proof}

\subsection{Arithmetic 2-friezes for $n=6$}\label{n=6List}

The classification in this case is as follows.

\begin{prop}
\label{n=6Prop}
The following 5 patterns:
\begin{equation}
\label{SixOne}
\begin{array}{rrrrrrrrrrrrr}
1&1&1&1&1&1&1&1&1&1&1&1\\
2&2&2&2&2&2&2&2&2&2&2&2\\
2&2&2&2&2&2&2&2&2&2&2&2\\
1&1&1&1&1&1&1&1&1&1&1&1
\end{array}
\end{equation}
\begin{equation}
\label{SixTwo}
\begin{array}{rrrrrrrrrrrrr}
1&1&1&1&1&1&1&1&1&1&1&1\\
1&3&5&2&1&3&5&2&1&3&5&2\\
5&2&1&3&5&2&1&3&5&2&1&3\\
1&1&1&1&1&1&1&1&1&1&1&1
\end{array}
\end{equation}
\begin{equation}
\label{SixThree}
\begin{array}{rrrrrrrrrrrrr}
1&1&1&1&1&1&1&1&1&1&1&1\\
1&1&2&4&4&2&1&1&2&4&4&2\\
1&1&2&4&4&2&1&1&2&4&4&2\\
1&1&1&1&1&1&1&1&1&1&1&1
\end{array}
\end{equation}
\begin{equation}
\label{SixFour}
\begin{array}{rrrrrrrrrrrrr}
1&1&1&1&1&1&1&1&1&1&1&1\\
1&1&3&6&3&1&1&2&3&3&3&2\\
1&2&3&3&3&2&1&1&3&6&3&1\\
1&1&1&1&1&1&1&1&1&1&1&1
\end{array}
\end{equation}
\begin{equation}
\label{SixFive}
\begin{array}{rrrrrrrrrrrrr}
1&1&1&1&1&1&1&1&1&1&1&1\\
1&1&4&6&2&1  & 2&3&2&2&4&3 \\
2&3&2&2&4&3 &1&1&4&6&2&1\\
1&1&1&1&1&1&1&1&1&1&1&1
\end{array}
\end{equation}
is the complete (modulo dihedral symmetry)
list of $12$-periodic arithmetic $2$-frieze patterns.
\end{prop}

\begin{proof}
We sketch an elementary, albeit somewhat tedious, proof. 

Let $a$ be the greatest common divisor of $x_1, y_2$, and $b$ that of $x_2, y_1$. Then 
$$
x_1=a \bar x_1,
\quad  y_2=a \bar y_2,
\quad x_2=b \bar x_2,
\quad y_1=b \bar y_1
$$
where the pairs $\bar x_1, \bar y_2$ and $\bar x_2, \bar y_1$ are coprime. 
Set: $p=x_1+y_2,\, q=x_2+y_1$.

Consider Example \ref{ClustSix}. From the third column we see that $q=kx_1=\ell{}y_2$ 
for some $k, \ell\in \Z_+$. 
Hence $k=A \bar y_2,\, \ell=A \bar x_1$ for $A\in \Z_+$, and $q=Aa\bar x_1 \bar y_2$. 
Likewise, $p=Bb\bar x_2 \bar y_1$. Thus
$$
a(\bar x_1+\bar y_2)=Bb\,\bar x_2 \bar y_1,
\qquad
b(\bar x_2+\bar y_1)=Aa\,\bar x_1 \bar y_2.
$$
Multiply these two equations, cancel $ab$, and rewrite in an equivalent form:
\begin{equation}
\label{MasterEq}
AB=\left(\frac{1}{\bar x_2} + \frac{1}{\bar y_1}\right) \left(\frac{1}{\bar x_1} + \frac{1}{\bar y_2}\right). 
\end{equation}
This Diophantine equation has just a few solutions, and this leads to the desired classification. 

Before we list the solutions of (\ref{MasterEq}), let us remark that the 4th and 5th columns of the 2-frieze in Example \ref{ClustSix} consist of the following integers:
$$
\frac{AB\bar x_2}{a},
\quad
\frac{AB\bar y_1}{a},
\quad
\frac{AB\bar x_1}{b},
\quad
\frac{AB\bar y_2}{b}.
$$
Therefore, once $A,B, \bar x_1, \bar x_2, \bar y_1, \bar y_2$ are found, one can determine the denominators, $a$ and $b$, by inspection.

Now we analyze equation (\ref{MasterEq}). First of all, at least one denominator must be equal to 1. If not, then the value of each parenthesis does not exceed $1/2+1/3=5/6$, and their product is less than 1. If 1 is present in the denominator in both parentheses then we have a Diophantine equation
$$
\left(1+\frac{1}{x}\right) \left(1+\frac{1}{y}\right) = AB\in\Z_+
$$
that, up to permutations, has the solutions $(1,1), (2,1)$ and $(2,3)$. The respective values of $AB$ are 4,\ 3 and 2. These solutions correspond to the 2-friezes (\ref{SixOne}) and (\ref{SixThree}), (\ref{SixFour}), and (\ref{SixFive}), respectively. 

If 1 is present in the denominator in only one parenthesis then we have a Diophantine equation
$$
\left(1+\frac{1}{x}\right) \left(\frac{1}{z}+\frac{1}{y}\right) = AB\in\Z_+
$$
where the second parenthesis does not exceed 5/6. It follows that $x\in \{1,2,3,4,5\}$. A case by case consideration yields one more solution: $x=5$ and $\{y,z\}=\{2,3\}$. The respective value of $AB$ is 1, and this corresponds to the 2-frieze (\ref{SixTwo}).
\end{proof}

\begin{rem}
Note that we have listed 2-friezes only up to dihedral symmetry. 
To be consistent with the case of the Coxeter-Conway friezes, where the count is given by the 
Catalan numbers, one should count the cases separately, that is, 
not to factorize by the dihedral group or its subgroups. 
Then the number of 2-friezes for $n=4,5,6,7$ is as follows: $1,5,51,868$. 
These numbers appeared in \cite{Pro}. 
We have independently verified this using an applet created by R. Schwartz for this purpose. 
The 2-frieze pattern of Proposition \ref{n=5Prop} gives 5 different patterns. 
The patterns of Proposition \ref{n=6Prop} contribute 1,8,6,12 and 24 different patterns, respectively. 
We do not have a proof that 868 is the correct answer, nor can we prove that the number of 
arithmetic 2-friezes is finite for each $n$. 
We hope to return to this fascinating combinatorial problem in the near future. 
Curiously, the only appearance of the sequence  $1,5,51,868$ in 
Sloane's Online Encyclopedia \cite{Slo} is in connection with Propp's paper \cite{Pro}.
\end{rem}

\subsection{One-point stabilization procedure}\label{StabOneSec}

Below we describe a procedure that allows one to obtain 2-frieze patterns
of width $m+1$ from 2-frieze patterns of width $m$.
More precisely, we consider $2n$-periodic 2-frieze patterns
whose first non-trivial row contains two consecutive entries 
equal to $1$.
Such a pattern can be obtained from a $2(n-1)$-periodic pattern
and, in this sense, may be considered ``trivial''.

\begin{prop}
\label{StabProp}
Let
$$
\ldots b_1,\;a_1,\;b_2,\;a_2,\;b_3,\;a_3\ldots
$$
be the first non-trivial row that generates a
$2n$-periodic arithmetic 2-frieze as in \eqref{DefClo}.
Then the frieze with the first non-trivial row
\begin{equation}
\label{StabRow}
\ldots{}b_n,\; a_n,\;b_1+1,\;a_1+b_2+1,\;
b_2+1,\;1,\;1,\;a_2+1,\;b_3+a_2+1,\;a_3+1,\;b_4,\;a_4\ldots
\end{equation}
is a $2(n+1)$-periodic arithmetic 2-frieze.
\end{prop}

In other words, we cut the line between $b_2$ and $a_2$, add $1,\,1$ and change
the three left neighbours: 
$$
(b_1,\,a_1,\,b_2)\to(b_1+1,\;a_1+b_2+1,\;b_2+1)
$$
and similarly with the three right neighbours.
The other entries remain unchanged.

\begin{proof}
Let $(V_i)$ be a solution to the difference equations \eqref{recur}
associated to the initial frieze.
We assume that $V_i\in\R^3$.
We wish to add an extra point $W\in\R^3$
so that the points
$$
\begin{array}{rcll}
\widetilde{V}_i&=&V_i,&i\leq n,\\[4pt]
\widetilde{V}_{n+1}&=&W,&\\
\end{array}
$$
give a solution to the difference equation
$$
\widetilde{V}_i=\widetilde{a}_i\,\widetilde{V}_{i-1}-
\widetilde{b}_i\,\widetilde{V}_{i-2}+\widetilde{V}_{i-3}.
$$
Geometrically speaking, we replace the $n$-gon
$\{V_1,\ldots,V_n\}$ by the $(n+1)$-gon
$\{V_1,\ldots,V_n,W\}$.

It is easy to check that the choice of $W$ is unique:
\begin{equation}
\label{ExtraPt}
W=(b_{2}+a_1+1)\,V_{n}-(b_1+1)\,V_{n-1}+V_{n-2}.
\end{equation}
The coefficients of the resulting equation are as follows
$$
\begin{array}{cccccccc}
\widetilde{b}_{n+1}
&\widetilde{a}_{n+1}&
\widetilde{b}_1&\widetilde{a}_1&\widetilde{b}_2&\widetilde{a}_2
&\widetilde{b}_3&\widetilde{a}_3\\[2pt]
\shortparallel&\shortparallel& \shortparallel& \shortparallel
& \shortparallel&
\shortparallel& \shortparallel& \shortparallel\\[2pt]
b_1+1&a_1+b_2+1&
b_2+1&1&1&a_2+1&b_3+a_2+1&a_3+1\\
\end{array}
$$
while $\widetilde{b}_i=b_i$ and $\widetilde{a}_i=a_i$
for $4\leq{}i\leq{}n$.
This corresponds to \eqref{StabRow}.

The frieze $F(\widetilde{a}_i,\widetilde{b}_i)$ generated by  \eqref{StabRow} is again integral.
Indeed, the entries of this frieze are polynomials in $\widetilde{a}_i,\widetilde{b}_i$,
see formula \eqref{DetTab}.
It remains to prove positivity of the frieze $F(\widetilde{a}_i,\widetilde{b}_i)$.

In the frieze $F(\widetilde{a}_i,\widetilde{b}_i)$, we choose two
consecutive diagonals $\Db_1$ and $\Db_{\frac{3}{2}}$.
Their entries are $\widetilde{v}_{i,1}$ and
$\widetilde{v}_{i+\half,\frac{3}{2}}$, respectively,
where $1\leq{}i\leq{}n-3$.
According to formula \eqref{MalDetEq}, one has:
$$
\widetilde{v}_{i,1}=
\left|
V_n,\; V_{n-1},\;V_i
\right|,
\qquad
\widetilde{v}_{i+\half,\frac{3}{2}}=
\left|
V_{i},\; V_{i+1},\; V_n
\right|.
$$
Therefore, these entries do not depend on $\widetilde{V}_{n+1}=W$
and, furthermore, all these entries belong to the initial frieze $F(a_i,b_i)$.
Hence, $\widetilde{v}_{i,1}$ and
$\widetilde{v}_{i+\half,\frac{3}{2}}$ are positive integers.

Finally, according to the rule of 2-friezes,
the diagonals $\Db_1$ and $\Db_{\frac{3}{2}}$ 
determine the rest and, moreover, all the entries are positive, see Theorem \ref{MainSecond}.
\end{proof}

\begin{rem}
It is clear that in the above stabilization process,
one can cut the first non-trivial line of $F(a_i,b_i)$ at an arbitrary place
(and not only between $b_2$ and $a_2$).
\end{rem}

Let us describe the geometry of one-point stabilization. 
The new point, $W$, is inserted between $V_n$ and $V_1$. 
One has the relation
$$
V_2=a_2V_1-b_2V_n+V_{n-1};
$$
it follows that 
$$
a_2V_1-V_2=b_2V_n-V_{n-1} =:U.
$$
One can easily check that
\begin{equation} \label{WU}
W=U+V_n+V_1.
\end{equation}
It follows from the definition of $U$ that
$$
\det (V_{n-1}, V_n, U) = \det (V_1, V_2, U)=0,\ \det (V_{n-2},V_{n-1},U)=b_2>0,\ \det (V_2,V_3,U)=a_2>0.
$$
Hence the vector $U$ belongs to the intersection of the two planes spanned by the pairs of vectors 
$(V_{n-1}, V_n)$ and $(V_1, V_2)$. 
Furthermore, $U$ is on the positive side of the two planes spanned by the pairs of vectors 
$(V_{n-2},V_{n-1})$ and $(V_2,V_3)$. 
Using the same central projection as in the proof of 
Lemma \ref{conv}, we conclude that the $n-1$-gon 
$\dots V_{n-2}, V_{n-1}, U, V_2, V_3,\dots$ in the horizontal plane is convex, 
see Figure \ref{PolyU}. 
This implies the inequalities $\det (V_{i-1},V_i, U) >0$ for $i\ne n,1,2$. 
In view of \eqref{WU} and the convexity of the polygon $(V_j)$, 
these inequalities imply that $\det (V_{i-1},V_i, W) >0$.

\begin{figure}[hbtp]
\includegraphics[width=5cm]{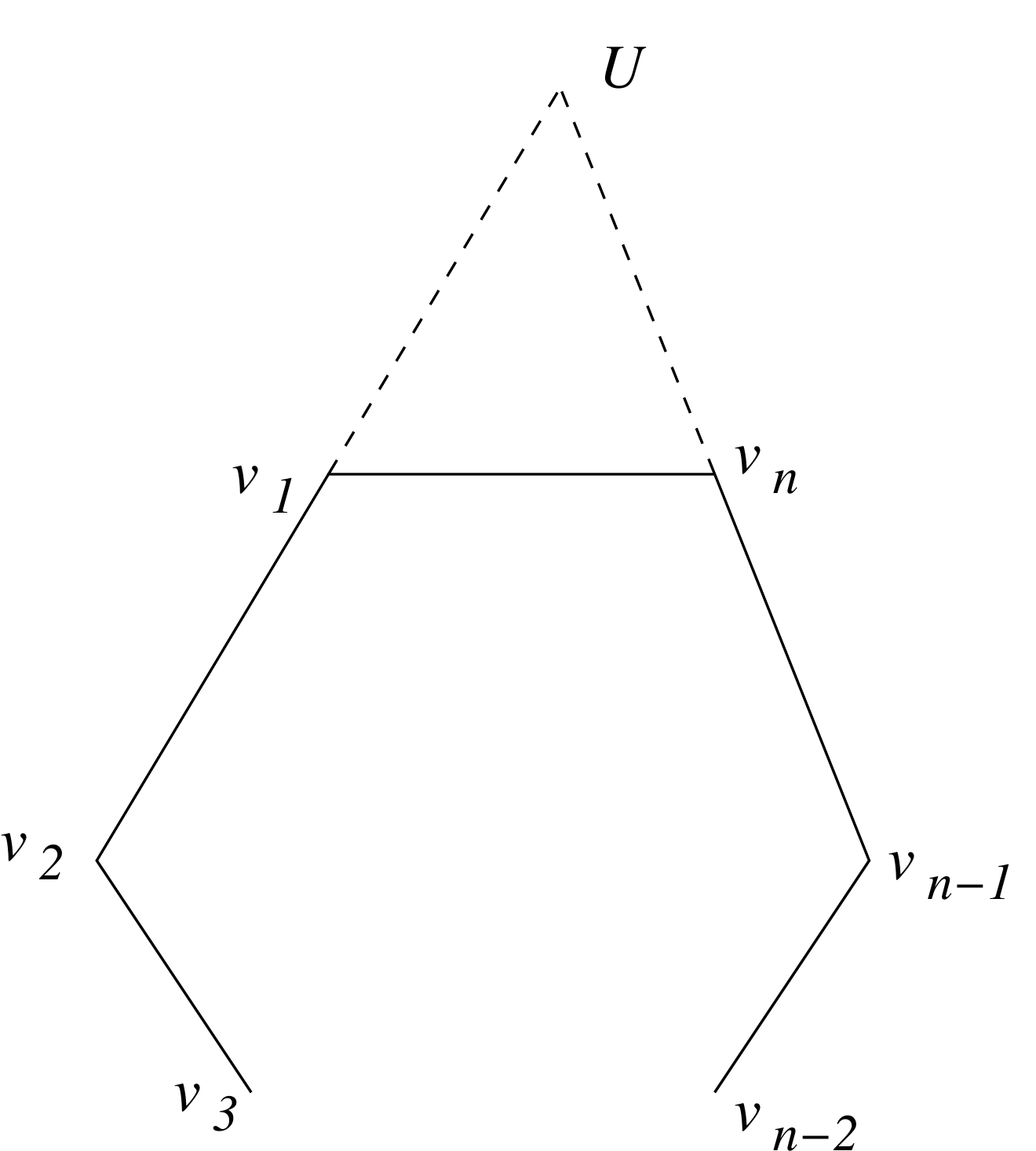}
\caption{Position of point $U$}
\label{PolyU}
\end{figure}

\medskip
The following statement is a reformulation of Proposition \ref{StabProp}.

\begin{cor}
An arithmetic 2-frieze pattern of width $m\geq1$ can be obtained via one-point stabilization from a pattern of width $m-1$ if and only if the second row
$(b_i,a_i)$ contains two consecutive ones.
\end{cor}

\begin{ex}
\label{NonStab}
(a)
The only $10$-periodic 2-frieze pattern \eqref{OldFriend} is obtained
by stabilization from the most elementary pattern \eqref{Triv}.
In this sense, there are no non-trivial $10$-periodic integral patterns.

(b)
All the patterns of  width 2, see Section \ref{n=6List}, except the first and the second,
are obtained by stabilization from \eqref{OldFriend}.
One therefore is left with two non-trivial $12$-periodic integral patterns,
namely \eqref{SixOne} and \eqref{SixTwo}.
\end{ex}

\subsection{Connected sum}\label{StabTwoSec}

We are ready to analyze the general connected summation 
and to prove Theorem \ref{MultiStabProp}. Let us start with an example. 

\begin{ex}\label{ExConS}
Consider the connected sum of the pentagon \eqref{OldFriend} with the hexagon \eqref{SixOne}. 
Cut the first row of \eqref{OldFriend} as follows:
$ 112\,|\,3211232$, 
insert six 2's, and change the two triples of neighbors of the block of 2's 
as required to obtain a 16-periodic arithmetic 2-frieze pattern corresponding to an octagon:
$$
\begin{array}{rrrrrrrrrrrrrrrr}
1&1&1&1&1&1&1&1&1&1&1&1&1&1&1&1\\
3&7&4&2&2&2&2&2&2&5&10&3&1&2&3&2\\
11&5&10&6&2&2&2&2&8&15&5&7&5&1&1&7\\
8&15&5&7&5&1&1&7&11&5&10&6&2&2&2&2\\
2&5&10&3&1&2&3&2&3&7&4&2&2&2&2&2\\
1&1&1&1&1&1&1&1&1&1&1&1&1&1&1&1
\end{array}
$$
\end{ex}
\medskip

We now turn to the proof of Theorem \ref{MultiStabProp}. 
Let us start with the remark that the roles played by the patterns $F(a_i,b_i)$ and $F(a_i',b_i')$ in the definition of connected sum are the same: interchanging the two results in the same pattern. 

First, we  prove that the connected sum of two closed 2-frieze patterns is also closed. 
Let $(V_i)$ be an $n$-gon corresponding to the difference equation \eqref{recur}, 
and let $(U_j)$ be a $k$-gon corresponding to a similar equation with coefficients $a_j',b_j'$. 
Consider a new difference equation with coefficients
$$
\begin{array}{l}
B_1=b_1'+b_1,
\quad 
A_1=a_1'+a_1+b_1'b_2,
\quad 
B_2=b_2'+b_2,
\quad 
A_2=a_2', 
\quad 
B_3=b_3',
\quad 
A_3=a_3', 
\quad
\dots  
\\[4pt]
\dots, \quad 
B_{k-1}=b_{k-1}',
 \quad
 A_{k-1}=a_{k-1}'+a_2, 
 \quad 
 B_k=b_k'+b_3+a_k' a_2 , 
 \quad 
 A_k=a_k'+a_3.
 \end{array}
$$
A solution to this equation is a sequence of points $W_m$ in $\R^3$; 
we may choose $W_{-2}, W_{-1}, W_{0}$ to be the standard basis. 
The polygon $(W_m)$ is  twisted: 
one has $W_{m+k} = M\left(W_m\right)$ for all $m$. 
The linear transformation $M$ is the monodromy of $W_m$. 

Assume that the vectors $V_{n-2}, V_{n-1},V_n$ also constitute the standard basis (this can be always achieved by applying a transformation from $\SL_3$). 

\begin{lem} 
\label{gluing}
The transformation $M$ takes $V_{n-2}, V_{n-1},V_n$ to $V_1,V_2,V_3$.
\end{lem}

\begin{proof}
Let us start with some generalities about difference equations and their monodromies  (see \cite{OST} for a detailed discussion).
Consider the difference equation \eqref{recur}. Let us construct its solution $V_i$ choosing the initial condition  $V_{-2}, V_{-1}, V_{0}$ to be the standard basis in $\R^3$. This is done by building a $3\times\infty$ matrix in which each next column is a linear combination of the three previous ones, as prescribed by  \eqref{recur}:
\begin{equation} \label{longmat}
\left(\begin{array}{lllllll}
1&0&0&1&a_2&a_2a_3-b_3&\dots\\[4pt]
0&1&0&-b_1&-b_1a_2+1&-b_1a_2a_3+a_3+b_1b_3&\dots\\[4pt]
0&0&1&a_1&a_1a_2-b_2&a_1a_2a_3-b_3a_1-a_3b_2+1&\dots
\end{array}\right).
\end{equation}
Three consecutive columns in the matrix \eqref{longmat} are also given by the product
$N_1N_2\dots N_r$ of $3\times 3$~matrices of the form
\begin{equation}
\label{threebythree}
N_j=\left(\begin{array}{ccc}
0&0&1\\
1&0&-b_j\\
0&1&a_j
\end{array}\right).
\end{equation}

With this preparation, we can compute the monodromy $M$ of the twisted polygon $(W_m)$. 
Thus $M$ is given as a product of matrices as in \eqref{threebythree} where all matrices, 
except the first two and the last two, are the same as for the polygon $(U_j)$.
The product of the first two is:
\begin{equation} 
\label{mat1}
\begin{array}{l}
\left(\begin{array}{ccc}
0&0&1\\[4pt]
1&0&-B_1\\[4pt]
0&1&A_1
\end{array}\right) 
\left(\begin{array}{ccc}
0&0&1\\[4pt]
1&0&-B_2\\[4pt]
0&1&A_2
\end{array}\right)
\\[26pt]
\qquad\qquad\qquad
=\left(\begin{array}{lll}
0&1&a_2'\\[4pt]
0&-b_1-b_1'&-b_1a_2'+1-b_1'a_2'\\[4pt]
1&a_1+b_2b_1'+a_1'&a_1a_2'-b_2+b_1'a_2'b_2+a_1'a_2'-b_2'
\end{array}\right),
\end{array}
\end{equation}
and the product of the last two is:
\begin{equation} 
\label{mat2}
\begin{array}{l}
\left(\begin{array}{ccc}
0&0&1\\[4pt]
1&0&-B_{k-1}\\[4pt]
0&1&A_{k-1}
\end{array}\right) 
\left(\begin{array}{ccc}
0&0&1\\[4pt]
1&0&-B_{k}\\[4pt]
0&1&A_k
\end{array}\right)
\\[26pt]
\qquad\qquad\qquad
=
\left(\begin{array}{lll}
0&1&a_k'+a_3\\[4pt]
0&-b_{k-1}'&1-b_{k-1}'a_k'-b_{k-1}'a_3\\[4pt]
1&a_{k-1}'+a_2&a_k' a_{k-1}'+a_{k-1}'a_3+a_2a_3-b_k'-b_3
\end{array}\right).
\end{array}
\end{equation}
Next, we observe that the matrices \eqref{mat1} and \eqref{mat2} decompose as
$$
\left(\begin{array}{ccc}
1&0&0\\[4pt]
-b_1&1&0\\[4pt]
a_1&-b_2&1
\end{array}\right)
\left(\begin{array}{ccc}
0&1&a_2'\\[4pt]
0&-b_1'&1-b_1'a_2'\\[4pt]
1&a_1'&a_1'a_2'-b_2'
\end{array}\right)
$$
and
$$
\left(\begin{array}{ccc}
0&1&a_k'\\[4pt]
0&-b_{k-1}'&1-b_{k-1}'a_k'\\[4pt]
1&a_{k-1}'&a_{k-1}'a_k'-b_k'
\end{array}\right)
\left(\begin{array}{ccc}
1&a_2&a_2a_3-b_3\\[4pt]
0&1&a_3\\[4pt]
0&0&1
\end{array}\right)
$$
respectively.
Thus $M$ is the product of $k+2$  matrices, and the product of the ``inner" $k$ of them is the monodromy of the closed $k$-gon $(U_j)$, that is, the identity matrix. What remains is the product of the first and the last matrices:
$$
\begin{array}{l}
\left(\begin{array}{ccc}
1&0&0\\[4pt]
-b_1&1&0\\[4pt]
a_1&-b_2&1
\end{array}\right)
\left(\begin{array}{lll}
1&a_2&a_2a_3-b_3\\[4pt]
0&1&a_3\\[4pt]
0&0&1
\end{array}\right)
\\[26pt]
\qquad\qquad\qquad
=
\left(\begin{array}{lll}
1&a_2&a_2a_3-b_3\\[4pt]
-b_1&-b_1a_2+1&-b_1a_2a_3+a_3+b_1b_3\\[4pt]
a_1&a_1a_2-b_2&a_1a_2a_3-b_3a_1-a_3b_2+1
\end{array}\right).
\end{array}
$$
The last matrix is the fourth 3 by 3 minor in \eqref{longmat}, that is, it takes $V_{-2}, V_{-1}, V_0$ to $V_1,V_2,V_3$, as claimed.
\end{proof}

Due to Lemma \ref{gluing}, the connected summation under consideration is the following procedure: arrange, by applying a volume preserving linear transformation, 
that the vertices  $W_{-2}, W_{-1}, W_{0}$ of 
 the twisted polygon $(W_m)$  coincide with $V_{n-2}, V_{n-1},V_n$, and insert $k-3$ vertices $W_1,W_2,\dots,W_{k-3}$ between $V_n$ and $V_1$. By Lemma \ref{gluing}, the vertices $W_{k-2}, W_{k-1}, W_k$ will coincide with $V_1,V_2,V_3$. 
 Thus a segment of length $k+3$ of the twisted polygon $(W_m)$ is pasted onto the polygon $(V_i)$ over coinciding triples of vertices on both ends. We have constructed a closed $(n+k-3)$-gon 
$$
\left\{
W_1,\, W_2,\dots,W_{k-3},\,V_1,\,V_2,\dots,V_n
\right\}
$$
satisfying the difference equation with coefficients as described in Theorem \ref{MultiStabProp}. 

Now we need to show that the connected sum of two arithmetic 2-frieze patterns is arithmetic as well. The argument is similar to the proof of Proposition \ref{StabProp}. The entries of the new pattern are polynomials in the entries of the first row, hence, integers. It remains to show that they are positive. For that purpose, we show there is a positive double zig-zag and refer to the positivity of Theorem \ref{MainSecond}.

Consider the $(n+k-3)$-gon corresponding to the connected sum, and assume that its vertices 
labeled $1$ through $n$ are the vertices of the $n$-gon 
$V_1,\dots,V_n$. Let $V_{n+1},\dots,V_{n+k-3}$ be the remaining vertices. 
Consider the consecutive diagonals $\Db_{n+2}$ and $\Db_{n+\frac{5}{2}}$. 
According to formula~\eqref{MalDetEq}, the entries of these two diagonals are 
$|V_i,V_{n-1},V_n|$ and $|V_i,V_{i+1},V_n|$, respectively. 
For $i=1,2,\dots,n-3$, these determinants are positive because the points involved are 
vertices of a convex $n$-gon $(V_j)$.

We claim that  $|V_i,V_{n-1},V_n|$ and $|V_i,V_{i+1},V_n|$ are also positive for $i=n+2,n+3,\dots,n+k-3$. Indeed, reversing the roles of the $n$-gon and $k$-gon in the construction of connected sum, we may assume that the $k$ consecutive points $V_{n-1},V_n,\dots,V_{n+k-3},V_1$ are the vertices of a convex $k$-gon $(U_j)$. This yields the desired positivity.

Theorem \ref{MultiStabProp} is proved. 

\begin{rem}
The procedure of connected sum can be understood directly from the friezes
as a vertical gluing of two friezes.
More precisely, the connected sum consists in choosing two consecutive columns in each frieze
and connecting them on the pair $1\; 1$. 
The connected two columns give two consecutive columns in the new frieze.
For instance in Example \ref{ExConS}, the new frieze is obtained by connecting
the columns\\
$$
\begin{matrix}
1&1\\
2&2\\
2&2\\
1&1
\end{matrix}\qquad \text{and}
\qquad
\begin{matrix}
1&1\\
2&3\\
1&1
\end{matrix}
$$
of  the friezes \eqref{SixOne} and \eqref{OldFriend} respectively,
one of the top of the other.

This procedure do not allow to obtain all the arithmetic friezes of a given width.
However, there exists a  more general procedure,
for which the gluing is not necessarily on a pair of ones,
that allows to construct more friezes. 
This procedure will be described in a separate work.
\end{rem}

\subsection{Examples of infinite arithmetic 2-frieze patterns}\label{ThreeEx}

In this section, we give examples of infinite  arithmetic 2-frieze patterns 
bounded above by a row of 1's and on the left by a double zig-zag of 1's.

\begin{ex}
In the following 2-frieze dots mean that the entries in the row stabilize.
$$
\setcounter{MaxMatrixCols}{20}
\begin{matrix}
1&1&1&1&1&1&1&1&1&1&1\\
&1&1&2&3&3&\dots\\
&&1&1&3&6&6&\dots\\
&&&1&1&4&10&10&\dots\\
&&&&1&1&5&15&15&\dots\\
&&&&&1&1&6&21&21&\dots
\end{matrix}
$$
The first two non-trivial South-East diagonals consist of consecutive positive integers and of consecutive binomial coefficients.
\end{ex}

\begin{ex}
In the next example, we choose two vertical arrays of 1's as the double zig-zag. As before, 
dots mean stabilization.
$$
\setcounter{MaxMatrixCols}{20}
\begin{matrix}
1&1&1&1&1&1&1&1&1&1&1&1\\
1&1&2&4&5&5&\dots\\
1&1&2&6&15&20&20&\dots\\
1&1&2&6&21&56&76&76&\dots\\
1&1&2&6&21&77&209&285&285&\dots\\
1&1&2&6&21&77&286&780&1065&1065&\dots\\
1&1&2&6&21&77&286&1066&2911&3976&3976&\dots
\end{matrix}
$$
The pattern is clear: all rows and all columns stabilize; the stabilization starts along two parallel South-East diagonals, and there is one other diagonal between the two, consisting of the numbers $1,4,15,56,209,780,2911,\dots$
The respective numbers in the two stabilizing diagonals differ by~1. 
It follows that the numbers on the diagonal between the two are the differences between the consecutive numbers on either of the stabilizing diagonals.

The numbers $d_n$ on the upper stabilizing diagonal $1,5,20,76,285,1065,3976,\dots$ satisfy the  relation
$$
d_{n+1}=\frac{d_n(d_n-1)}{d_{n-1}}
$$
that follows from the 2-frieze relation. One learns from Sloane's Encyclopedia
\cite{Slo} that these numbers also satisfy a linear recurrence
$$
d_{n+1}=4d_n-d_{n-1}+1,
$$
which can be easily proved by induction on $n$. Solving the above linear recurrence is standard.
\end{ex}

\begin{ex}
In the next example, the double zig-zag of 1's indeed looks like a zig-zag:
$$
\setcounter{MaxMatrixCols}{20}
\begin{matrix}
1&1&1&1&1&1&1&1&1&1&1\\
1&1&2&3&4&6&5&6&5&6&5\\
  &1&1&5&14&14&31&19&31&19&31\\
1&1&2&3&14&70&47&157&66&157&66\\
  &1&1&5&14&42&353&155&793&221&793\\
1&1&2&3&14&70&131&1782&507&4004&728\\
   &1&1&5&14&42&353&417&8997&1652&20216\\
1&1&2&3&14&70&131&1782&1341&45425&5373      
\end{matrix}
$$
In this 2-frieze pattern, the horizontal and vertical stabilization is different from the previous examples: each row and each column is eventually 2-periodic. There are five different South-East diagonals. Interestingly, they are all in Sloane's Encyclopedia \cite{Slo}. We list them here, along with their Sloane's numbers:
$$
\begin{array}{ll}
1,2,5,14,42,131,417,\dots & A080937;\\[4pt]
1,3,14,70,353,1782,8997,\dots & A038213;\\[4pt]
1,4,14,47,155,507, 1652,\dots & A094789;\\[4pt]
1,6,31,157,793,4004,20216,\dots & A038223;\\[4pt]
1,5,19,66,221,728,2380,\dots & A005021.
\end{array}
$$
\end{ex}

 \section{Appendix: Frieze patterns of Coxeter-Conway, difference equations, \\
 \hspace{-2cm} polygons, and the
 moduli space $\cM_{0,n}$}

In this appendix we review the classical case of Coxeter-Conway frieze patterns in their relation with second order difference equations, polygons in the plane and in the projective line, and the configuration space of  the projective line. We refer to \cite{Cox,CoCo} for information on frieze patterns; see also \cite{MS} and \cite{Tab} for details concerning some of our remarks.\footnote{We strongly recommend R. Schwartz's applet \url{http://www.math.brown.edu/~res/Java/Frieze/Main.html}}

As before, we consider the space $\cC_n$ of polygons in $\pP^1$, that is, $n$-tuples of cyclically ordered points $(v_i)$ such that $v_i\ne v_{i+1}$ for all $i$. Polygons in $\pP^1$ are considered modulo projective equivalence.  Let $\tilde \cC_n$ be the space of origin symmetric $2n$-gons $(V_i)$ in the plane satisfying the determinant condition $|V_i,V_{i+1}|=1$ for all $i$. Polygons in the plane are considered modulo $\SL_2$-equivalence.

Another relevant space is the moduli space $\cM_{0,n}$ of stable curves of genus zero with $n$ distinct marked points, 
 defined as the space of ordered $n$-tuples of points in $\CP^1$ modulo projective equivalence:
$$
\cM_{0,n}=
\left\{
(v_1,\ldots,v_n)\in\CP^1
\left|
v_i\not=v_j,\;i<j
\right.
\right\}
/\PSL(2,\C).
$$
The space $\cM_{0,n}$ is classical, and it continues to play an important role in the current research
(see, e.g., \cite{AFV}).
We show in this Appendix that $\cM_{0,n}$ is an open dense subset of a cluster manifold, 
provided $n$ is odd (this condition is a 1-dimensional counterpart to the condition that $n$
is not a multiple of 3 that we encountered earlier).
This cluster structure is closely related to that on the Teichmuller space, see~\cite{FG}, but it
is more difficult to construct.
We did not find an appropriate reference in the literature~\cite{Cha}.
We use the classical Coxeter-Conway friezes 
(with coefficients in $\C$). 
The space $\cM_{0,n}$ coincides with the subset of friezes such that all the entries are different from 0. 
This observation is rather simple but we did not find it explicitly in the literature. Two immediate consequences are as follows.

\begin{enumerate}
\item
One obtains several natural coordinate systems on $\cM_{0,n}$, one of which is compatible with a cluster structure. More precisely, $\cM_{0,n}$ is a smooth cluster manifold of type~$A_{n-3}$.
\item
Many objects related to $\cM_{0,n}$, such as discrete versions of KdV, etc., can be formulated in terms of Coxeter-Conway friezes.
\end{enumerate}

\medskip

 \paragraph*{\bf Space $\tilde \cC_n$, difference equations and Coxeter-Conway friezes}

We consider the following, infinite and row $n$-periodic, frieze pattern:
$$
\begin{array}{ccccccccccc}
\cdots&&1&&1&&1&&1&&\cdots\\[4pt]
&C_i&&C_{i+1}&&C_{i+2}&&C_{i+3}&&C_{i+4}\\[4pt]
&& \cdots&& \cdots&& \cdots&& \cdots&&
\end{array}
$$
where $C_i(=C_{i+n})$ are formal variables
and where all the entries are polynomials determined by the  first row via the frieze rule
$AD-BC=1$, for each elementary square:
$$
\begin{array}{ccc}
&B&\\[4pt]
A&&D\\[4pt]
&C&
\end{array}
$$
For instance, the entries in the next row are:
$C_iC_{i+1}-1$, etc.
As before, a numerical frieze $F(c_i)=F(C_i)|_{C_i=c_i}$ is obtained by evaluation.

\begin{rem}
In order to give a correct definition of space of friezes,
one has to adapt the technique
of algebraic friezes and treat $C_i$ as formal variables.
Otherwise, the frieze rule does not suffice to determine the entries
of the pattern (if too many of $c_i$ vanish), cf. Section \ref{ANuF}.
\end{rem}

An $n$-periodic frieze pattern is closed if  it contains a row of $1$'s (followed by a row of~$0$'s).
$$
\begin{array}{ccccccccccc}
\cdots&&1&&1&&1&&1&&\cdots\\[4pt]
&c_i&&c_{i+1}&&c_{i+2}&&c_{i+3}&&c_{i+4}\\[4pt]
&& \cdots&& \cdots&& \cdots&& \cdots&&\\[4pt]
\cdots&&1&&1&&1&&1&&\cdots
\end{array}
$$
The width (the number of non-trivial rows) of the
above frieze pattern is equal to $n-3$, see \cite{CoCo}.

One associates a second order difference equation with periodic coefficients with a closed frieze pattern: 
\begin{equation}
\label{MSTS}
V_{i+1}=c_i \,V_{i}-V_{i-1}; \qquad c_{i+n}=c_i.
\end{equation}
We understand its solutions $(V_i)$ as vectors in the plane satisfying the relation $|V_i,V_{i+1}|=1$. Equation \eqref{MSTS} determines the polygon $(V_i)$ uniquely, up to $\SL_2$-action.

We label $(v_{i,j})_{i,j\in \Z}$ the entries of the frieze, such that $v_{i,i}=c_i$, and according to the scheme:
$$
\begin{array}{ccc}
&v_{i,j}&\\[4pt]
v_{i,j-1}&&v_{i+1,j}\\[4pt]
&v_{i+1,j-1}&
\end{array}
$$
Analogs of Proposition \ref{TabProp}, Proposition \ref{proprec} and Lemma \ref{DetLem} hold true providing explicit formul{\ae}. Namely, one has: 
$$
v_{i,j}=|V_i,V_j|,
$$
and
\begin{equation} \label{matdet}
v_{i,j}=\left|\begin{array}{cccccc}
c_{j}&1&&&\\\
1&c_{j+1}&1&&\\
&\ddots&\ddots&\ddots&\\
&&1&c_{i-1}&1\\
&&&1&c_{i}
\end{array}\right|.
\end{equation}

As a consequence of these formul{\ae}, $V_{i+n}=-V_i$ for all $i$, that is, the monodromy of 
equation~\eqref{MSTS} is $-\Id\in\SL_2$. 
This provides the equivalence between closed frieze patterns and the space of polygons $\tilde \cC_n$.

\medskip 

 \paragraph*{\bf Polygons in the plane and in the projective line}

As before, one has a natural projection $\tilde \cC_n\to \cC_n$ from~$\R^2$ to $\RP^1$. 
If $n$ is odd then this is a bijection, cf. Section \ref{relations}. This provides an equivalence between projective equivalence classes of $n$-gons in $\CP^1$ and $SL_2$-equivalence classes of origin symmetric $2n$-gons $\C^2$, subject to the unit determinant condition. 

Note that, over reals, there is an additional obstruction to lifting a polygon from $\RP^1$ to~$\R^2$. 
Let $n$ be odd and $(v_i)$ be a polygon in the projective line. Let $V_i\in \R^2$ 
be some lifting of points~$v_i$. 
For the system of equations
$$
t_it_{i+1} = 1/|V_i,V_{i+1}|,\ i=1,\dots,n-1,\quad t_1t_n=1/|V_1,V_n|
$$
to have a real solution, one needs $\Pi_{i=1}^n |V_i,V_{i+1}| >0$.  If this condition holds then $t_i$ is uniquely determined, up to a common sign; otherwise there is a lifting satisfying the opposite condition $|V_i,V_{i+1}|=-1$ for all $i$.

\medskip 

 \paragraph*{\bf Cluster coordinates}

The space of friezes has another natural coordinate system
apart from~$c_i$.
Unlike the coordinates $c_i$ that satisfy  three very non-trivial
equations given by the condition that the frieze pattern is closed,
the new coordinates are free. These three conditions are as follows:
$$
v_{0,n-1}=1,\qquad v_{-1,n-1}=0,\qquad v_{0,n}=0
$$
where $v_{i,j}$ are given by the determinants \eqref{matdet} (the fourth condition, $v_{-1,n}=-1$, follows from the fact that the monodromy is area-preserving).

An arbitrary zig-zag (i.e. piecewise linear path from top to bottom such that each segment goes either down-right or down-left) filled by the variables 
$x_1, \ldots{},x_{n-3}$
$$
 \begin{array}{ccccccc}
1&& 1&&1&&\cdots
 \\[4pt]
&x_1&&\cdots&&&
 \\[4pt]
&&x_2&&\cdots&&
 \\[4pt]
 &x_3&&\cdots&&&
 \\[4pt]
  x_4&&\cdots&&&&
 \\[4pt]
&\cdots &&\cdots & &&\\[4pt]
1&&1&&1&&\cdots
\end{array}
$$
determines the rest of the pattern.
The subalgebra of $\C(x_1, \ldots{},x_{n-3})$ generated by all the rational functions arising in the pattern
is the cluster algebra
associated to the quiver
of type $A_{n-3}$, see~\cite{CaCh}.
The initial zig-zag $(x_1, \ldots{},x_{n-3})$ forms the initial cluster,
and to different zig-zags correspond different clusters. 
However, some clusters are not obtained as zig-zags in the pattern.
In type  $A_{n-3}$ there is a correspondence between clusters and triangulations of a $n$-gon
\cite{FZ1} (see also \cite{Sco}).
The clusters which are not zig-zags in the frieze correspond to triangulations containing inner triangles
(i.e. triangles built on three diagonals). 

One then constructs a smooth cluster manifold gluing together
the tori $(\C^*)^{n-3}$ according to the coordinate changes defined
by consecutive mutations.

\begin{prop}
\label{LastProp}
The space $\cM_{0,n}$ is a smooth submanifold of the constructed
cluster manifold.
\end{prop}

\begin{proof}
The fact that $v_i\not=v_j$ for all $1\leq{}i<j\leq{}n$ in the definition of $\cM_{0,n}$,
is equivalent to the fact that all the entries of the corresponding frieze
$v_{i,j}\not=0$.
Therefore, the points of $\cM_{0,n}$ are non-singular in any chart.
\end{proof}

\begin{ex}
If $n=5$, then one has
$$
 \begin{array}{cccccccc}
&1&& 1&&1&&\cdots
 \\[4pt]
x_1&&\frac{x_2+1}{x_1}&&\frac{x_1+1}{x_2}&&x_2&\cdots
 \\[4pt]
&x_2&&\frac{x_1+x_2+1}{x_1x_2}&&x_1&&\cdots
 \\[4pt]
1&&1&&1&&1&\cdots
\end{array}
$$
which correspond to the $A_2$-case, quite similarly to
Example \ref{ClustFive}.
\end{ex}

The cluster structure on $\cM_{0,n}$, with $n=2m+1$, that we have just constructed,
coincides with that communicated to us by F. Chapoton \cite{Cha}.
 
 \bigskip

\noindent \textbf{Acknowledgements}.
We are grateful to the Research in Teams program at BIRS
where this project was started.
We are pleased to thank Ph. Caldero, F. Chapoton, B. Dubrovin, V. Fock,
B.~Keller, R. Kenyon, I. Krichever, J. Propp, D.~Speyer, Yu. Suris, and A. Veselov
 for interesting discussions. 
Our special gratitude goes to R. Schwartz for numerous fruitful discussions and help with computer experiments.

\vskip 1cm

\end{document}